\documentclass[final,leqno]{siamltex}
 \usepackage{threeparttable}
 \usepackage{dcolumn}
  \newcolumntype{d}{D{.}{.}{-1}}
\usepackage{subfigure}

\usepackage[german,english]{babel}
\usepackage{epsfig}
\usepackage{epstopdf}
\usepackage{amssymb}
\usepackage{algorithm}
\usepackage{algorithmic}
\usepackage{color}
\usepackage{graphicx}
\usepackage{booktabs}
\usepackage{amsmath}
 \usepackage{amsfonts}
\usepackage{soul}

\title{An Online Manifold Learning Approach for  Model Reduction of Dynamical Systems\thanks{
This research was supported by the Office of Naval Research.}}

\author{Liqian Peng\thanks{Department of Mechanical and
Aerospace Engineering, and Institute for Networked Autonomous
Systems, University of Florida, Gainesville, FL 32611-6250
(liqianpeng@ufl.edu).}
        \and Kamran Mohseni\thanks{Department of Mechanical and
Aerospace Engineering, Department of Electrical and Computer
Engineering, and Institute for Networked Autonomous Systems,
University of Florida, Gainesville, FL 32611-6250
(mohseni@ufl.edu).}}

\begin{document}

\maketitle
\begin{abstract}
This article discusses a newly developed online manifold learning
method, subspace iteration using reduced models (SIRM), for the
dimensionality reduction of dynamical systems. This method may be
viewed as subspace iteration combined with a model reduction
procedure. Specifically, starting with a test solution, the method
solves a reduced model  to obtain   a more precise solution, and
it repeats this process  until  sufficient accuracy is achieved.
The reduced model is obtained by projecting the full model onto a
subspace that is  spanned by  the dominant modes of an extended
data ensemble. The extended data ensemble in this article contains
not only the  state vectors of some snapshots of the approximate
solution from the previous iteration but also the associated
tangent vectors.   Therefore, the proposed manifold learning
method takes advantage of the information of the original
dynamical system to reduce the dynamics. Moreover, the
learning procedure is computed in the online stage, as opposed to being
computed offline, which is used in many
projection-based model reduction techniques that require prior
calculations or experiments.
After providing an error bound of the classical POD-Galerkin method  in terms of the projection error and the initial condition error, we prove that the  sequence of approximate solutions converge to the actual solution of the original system as long as the vector field of the full model
is locally Lipschitz on an open set that contains the solution
trajectory. Good accuracy of the proposed method has been
demonstrated in two numerical examples, from  a linear
advection-diffusion equation to  a nonlinear Burgers equation.  In order to save
computational cost, the SIRM method is extended to a local model
reduction approach by partitioning  the entire time domain into
several subintervals and obtaining a series of local reduced
models of much lower dimensionality.  The accuracy and efficiency
of the local SIRM are shown through the numerical simulation of
the Navier--Stokes equation in a lid-driven cavity flow problem.
\end{abstract}

\begin{keywords}
online, manifold learning, subspace iteration, model reduction,
local model reduction
\end{keywords}

\begin{AMS}
78M34, 37M99, 65L99, 34C40, 74H15, 37N10
\end{AMS}

\pagestyle{myheadings} \thispagestyle{plain} \markboth{LIQIAN PENG
AND KAMRAN MOHSENI}{AN ONLINE MANIFOLD LEARNING APPROACH FOR MODEL
REDUCTION}

\section{Introduction}
The simulation, control, design, and analysis
of the methods and algorithms for many large-scale dynamical systems
are  often
computationally intensive and  require massive computing resources if at all possible.
  The idea of model reduction is to provide an
efficient computational prototyping tool to replace a high-order
system of differential equations with a system of a substantially
lower dimension, whereby only the most dominant properties of the
full system are preserved. During the past several decades,
several model reduction methods have been studied, such as Krylov
subspace methods~\cite{BaiZ:02a},  balanced truncation
\cite{MooreBC:81a, ScherpenJMA:93a, Marsden:02g}, and proper
orthogonal decomposition (POD)~\cite{MooreBC:81a, LoeveM:55a}.
More techniques can be found in ~\cite{SorensenDC:01a}
and~\cite{AntoulasAC:05a}. These model reduction methods are usually based on offline computations to build the empirical eigenfunctions of the reduced model  before the computation of the reduced state variables. Most of the time these offline computations are as complex as the original simulation.
 For these reasons, an efficient reduced model with high fidelity based
on online manifold learning is preferable.\footnote{To avoid
confusion, ``online'' in this article means that the manifold
learning is carried out in the online stage, as opposed to the
offline stage.} However, much less effort has been expended in the
field of model reduction via online manifold learning.
In ~\cite{RyckelynckD:97a},  an incremental
algorithm involving adaptive periods was proposed. During these adaptive periods the incremental computation is restarted until a
quality criterion is satisfied. In~\cite{MarkovinovicR:06a} and~\cite{RapunM:10a} state
vectors from  previous time steps are extracted to span a linear
subspace in order to construct the reduced model for the next
step. In ~\cite{PetzoldLR:02a} dynamic iteration using
reduced order models (DIRM) combines the idea of the waveform
relaxation technique and model reduction, which simulates each
subsystem that is connected to model reduced versions of the other
subsystems.

A new framework of iterative manifold learning, subspace iteration
using reduced models (SIRM),  is proposed in this article for the
reduced modeling of high-order nonlinear dynamical systems.
Similar to the well-known Picard iteration for solving ODEs, a
trial solution is set at the very beginning. Using POD, a set of updated
empirical eigenfunctions are constructed in each iteration by
extracting dominant modes from an extended data ensemble; then, a
more accurate solution is obtained by solving  the  reduced
equation in a new subspace spanned by these empirical
eigenfunctions. The extended data ensemble contains not only  the
state vectors of some snapshots of the trajectory in the previous
iteration but also the associated tangent vectors. Therefore, the
manifold learning process essentially takes advantage of the
information from the original dynamical system. Both analytical
results and numerical simulations indicate that a sequence of
functions asymptotically converges to the solution of the full
system. Moreover, the aforementioned method can be used to test
(and improve) the accuracy of a trial solution of  other
techniques. A posterior error estimation can be estimated by the
difference between the trial solution and a more precise solution
obtained by SIRM.

The remainder of this article is organized as follows.   Since
algorithms in this article fall in the category of projection
methods, the classic POD-Galerkin method and its ability to
minimize truncation error are briefly reviewed  in section
\ref{sec:background}.  After presenting the SIRM algorithm in
section \ref{sec:SIRM}, we provide
convergence analysis, complexity analysis, and two numerical examples. Then  SIRM
is combined with the time domain partition  in section
\ref{sec:llp}, and
 a local SIRM method is proposed to decrease redundant dimensions. The
performance of this technique is evaluated in a lid-driven cavity
flow problem. Finally, conclusions are offered.

\section{Background of Model Reduction} \label{sec:background}
Let $J=[0, T]$ denote the time domain,  $x: J \to \mathbb{R}^n$ denote the state variable, and  $f: J \times
\mathbb{R}^n \to \mathbb{R}^n$ denote the discretized vector field.
A  dynamical system  in $\mathbb{R}^n$ can be described by an initial value problem
\begin{equation} \label{general}
\dot x = f(t, x); \ \ \ \ \ x(0)=x_0.
\end{equation}
By definition, $x(t)$ is a flow that gives an orbit in $\mathbb{R}^n$
as $t$ varies over $J$ for a fixed $x_0$. The orbit contains a
sequence of states (or state vectors) that follow from $x_0$.

\subsection{Galerkin projection} For a $k$-dimensional linear subspace $S$ in $\mathbb{R}^n$, there
exists an $n \times k$ orthonormal matrix $ \Phi =[\phi_1, ...,
\phi_k] $, the columns of which form a complete basis  of $S$. The
orthonormality of the column matrix requires that $\Phi^T
\Phi=I$, where
 $I$ is an identity matrix. Any
state $x \in \mathbb{R}^n$ can be projected onto $S$ by a linear
projection. The projected state is given by $\Phi^Tx\in
\mathbb{R}^k$ in the subspace coordinate system, where
superscript $T$ denotes the matrix transpose. Let  $P:=\Phi \Phi^T$
denote the projection
 matrix in $\mathbb{R}^n$. Then, the same projection in the original
coordinate system is represented by  $\tilde x(t) :=P x(t)\in
\mathbb{R}^n$.

Let $\Phi^T f(t, \Phi z)$ denote a reduced-order vector field
 formed by  Galerkin projection.  The corresponding
reduced model for $z(t)\in \mathbb{R}^k$ is
\begin{equation} \label{rom}
\dot z  =   \Phi^T f(t, \Phi z); \ \ \ \ \    z_0=\Phi^T x_0.
\end{equation}
 An approximate solution in the original coordinate system   $\hat x(t)=\Phi z(t) \in
\mathbb{R}^n$ is equivalent to the solution of the following ODE:
\begin{equation} \label{rom1}
\dot {\hat x} = P f(t, \hat x); \ \ \ \ \ \hat x_0=P x_0.
\end{equation}

It is well-known that the existence and uniqueness of a solution
for system (\ref{general}) can be proved by the Picard iteration.
\medskip
\begin{lemma}[{\rm Picard--Lindel{\"o}f  existence and
uniqueness}~\cite{MeissJD:07a}] \label {existence} Suppose there
is a closed ball of radius $b$ around a point $x_0 \in
\mathbb{R}^n$ such that $f: J_a \times B_b(x_0) \to \mathbb{R}^n$
is a uniformly Lipschitz function of $x \in B_b(x_0)$ with
constant $K$, and a continuous function of $t$ on $J_a = [0, a]$.
Then the initial value problem \emph{(\ref{general})} has a unique
solution $x(t) \in B_b(x_0)$  for $t \in J_a$, provided that
$a=b/M$, where
\begin{equation} \label{definem}
M = \mathop {\max }\limits_{  (t,x) \in J_a\times B_b(x_0)}
\left\| {f(t,x)} \right\|.
\end{equation}
\end{lemma}

Similarly, a reduced model formed by Galerkin projection  also has
a unique local solution if the original vector field is Lipschitz.
Moreover, the existence and uniqueness of solutions do not depend
on the projection operator.

\medskip
\begin{lemma} [{\rm  local existence and
uniqueness of reduced models}] \label{localexist} With $a$, $J_a$,
$b$, $B_{b}(x_0)$, $M$, and $f(t,x)$ defined in Lemma
\ref{existence}, the reduced model \emph{(\ref{rom1})}  has a
unique solution $\hat x(t) \in B_b(x_0)$ at the interval $t\in
J_0=[0,a/2]$ for a given initial condition $\hat x(0)=\hat x_0$,
provided that $a=b/M$ and $\left\|\hat x_0-x_0\right\|<b/2$.
 \end{lemma}
 \medskip
\begin{proof}
Since $f(t,x)$ is a uniformly Lipschitz function of $x$ with
constant $K$ for all $(t,x) \in J_a \times B_b(x_0)$,  then
$$
\left\| {f(t, x_1) - f(t, x_2)} \right\| \le K\left\| {x_1 - x_2}
\right\|
$$
for  $x_1, x_2 \in B_b(x_0)$ with $t\in J_a$. Since $P$ is a
 projection matrix, $\left\| P \right\| =1$. As a
consequence,
$$
\left\| P f(t, x_1) - P f(t, x_2) \right\| \le \left\| P \right\|
\left\| {f(t, x_1) - f(t, x_2)} \right\| \le K\left\| {x_1 - x_2}
\right\|,
$$
 which justifies that the projected vector field $P f(t,x)$ is
 also Lipschitz with constant $K$ for the same domain. Since $\left\|\hat
 x_0-x_0\right\|<b/2$, we have $B_{b/2}(\hat x_0)\subset
 B_b(x_0)$.
By  Lemma \ref{existence}, $\dot {\hat x} = P f(t, \hat x)$
  has a uniquely local  solution $\hat x(t)
\in B_{b/2}(\hat x_0)$ for $t\in [0, a_1]  \cap J_a$, where
  $a_1$ is given by
$$
a_1 = \frac{b/2}{{\mathop {\max }\limits_{(t,x) \in {J_0} \times
{B_b}({x_0})} \left\| {\Phi f(t,x)} \right\|}}.
$$
Since $\left\| {\Phi f(t,x)} \right\| \le \left\| {f(t,x)}
\right\|$, we have $a_1 \ge b/2M=a/2$. Therefore, $J_0\subset
[0,a_1] \cap J_a$, and there exists a unique solution $\hat x(t)
\in B_b(x_0)$  for the interval $J_0$. \quad
\end{proof}
\medskip

The error of the reduced model formed by the Galerkin projection
can be defined as $e(t) := \hat x(t) -x(t)$. Let
$e_o(t):=(I_k-P)e(t)$, which denotes the error component
orthogonal to  $S$, and $e_i(t):=Pe(t)$, which denotes the
component of error parallel to $S$ (see Figure
\ref{fig:poderror}). Thus,  we have
\begin{equation}
 e_o(t)=\tilde x(t)-x(t),
 \end{equation} which directly comes from the
projection. However, since the system is evolutionary with time,
further approximations of the projection-based reduced model
result in an additional error $e_i(t)$, and  we have
\begin{equation}
e_i(t)=\hat x(t)-\tilde x(t).
\end{equation} Although $e_i(t)$ and $e_o(t)$ are orthogonal to
each other, they are not independent~\cite{PetzoldLR:03a}.
\begin{figure}
\begin{center}
\begin{minipage}{0.6\linewidth} \begin{center}
\includegraphics[width=1\linewidth]{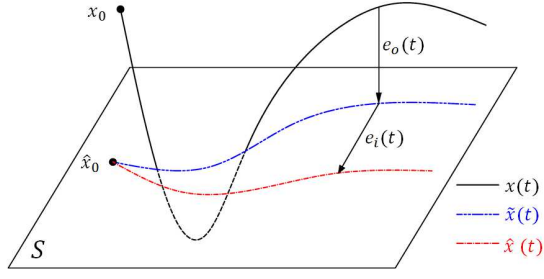}
\end{center} \end{minipage}\\
\caption{   Illustration of the actual solution
$x(t)$ for the original system (\ref{general}), the projected
solution $\tilde x(t)$ on $S$,  and the approximate solution $\hat
x(t)$ computed by the reduced model (\ref{rom1}). The component of
error orthogonal to $S$ is given by $e_o(t)=\tilde x(t)-x(t)$ and
the component of error parallel to $S$ is given by $e_i(t)=\hat
x(t)-\tilde x(t)$. This figure is reproduced from
\cite{PetzoldLR:03a}.} \label{fig:poderror}\vspace{-3mm}
\end{center}
\end{figure}

\medskip
\begin{lemma} \label{bound} Consider the initial value problem \emph{(\ref{general})} over the interval
$J_0=[0,a/2]$. $a$, $J_a$,
 $b$, $B_{b}(x_0)$, $M$, $P$, $x(t)$, $\tilde x(t)$,  $\hat
x(t)$, $e(t)$, $e_o(t)$, and $e_i(t)$ are defined as above. Suppose
$f(t,x)$ is a uniformly Lipschitz function of $x$ with constant
$K$ and a continuous function of $t$  for all $(t, x) \in J_a
\times B_b(x_0)$.  Then the error $e(t)=\hat x(t)-x(t)$ in the
infinity norm for the interval $J_0$ is bounded by
\begin{equation}\label{eeo}
\left\| e\right\|_\infty \le e^{Ka/2}  \left\| e_o \right\|_\infty
+e^{Ka/2} \left\| {e_i(0)} \right\|.
\end{equation}
\end{lemma}
\begin{proof}
Since $f(t,x)$  is a uniformly Lipschitz function for any
$(t,x)\in J_a \times B_b(x_0)$, Lemmas \ref{existence} and
\ref{localexist} respectively imply the unique existences of
$x(t)\in B_b(x_0)$ and  $\hat x(t)\in B_b(x_0)$. Moreover, we can
uniquely determine $\tilde x(t)\in B_b(x_0)$ by $\tilde x(t)=P
x(t)$. Therefore, $x(t)$, $\tilde x(t)$, and $\hat x(t)$ are all
well-defined for any $t\in J_0$.

Substituting (\ref{general}) and (\ref{rom1}) into the
differentiation of $e_o(t)+ e_i(t) = \hat x(t) - x(t)$ yields
\begin{equation}\label{ee}
\dot e_o + \dot e_i = Pf(t,\hat x) - f(t,x).
\end{equation}
Left multiplying (\ref{ee})  by $P$, expanding $\hat x$, and
recognizing that $P^2 = P$ gives
$$\label{ei}
 \dot e_i(t) = P(f(t, x + e_o+ e_i) -f(t,x)).
$$
 Using this equation by expanding  $\left \|e_i(t+h) \right \|$ and  applying triangular inequality yields
$$
\begin{array}{l}
 \left\| {{e_i}(t + h)} \right\| = \left\| {{e_i}(t) + hPf(t,x + {e_o} + {e_i}) - hPf(t,x)} \right\| + \mathcal{O}({h^2})
 \smallskip \\
 \ \ \ \ \ \ \ \ \ \ \ \ \ \ \  \le \left\| {{e_i}(t)} \right\| + h\left\| {Pf(t,x + {e_o} + {e_i}) - Pf(t,x + {e_o})} \right\| \smallskip \\
\ \ \ \ \ \ \ \ \ \ \ \ \ \ \ \ \ \ \ \ \ \ \ + h\left\| {Pf(t,x +
{e_o}) - Pf(t,x)} \right\| + \mathcal{O}({h^2}).
 \end{array}
$$
Rearranging this inequality and applying the  Lipschitz conditions
gives
$$
\frac{{\left\| {{e_i}(t + h)} \right\| - \left\| {{e_i}(t)}
\right\|}}{h} \le K \left\| {{e_i}(t)} \right\| + K \left\|
{{e_o}(t)} \right\| + \mathcal{O}(h).
$$
Since $\mathcal{O}(h)$ can be uniformly bounded independent of
$e_i(t)$, using the mean value theorem and letting $h \to 0$ give
$$
\frac{{\rm{d}}}{{{\rm{d}}t}}\left\| {{e_i}(t)} \right\| \le K\left\| {{e_i}(t)} \right\| + K\left\| {{e_o}(t)} \right\|.
$$
Rewriting the above inequality into integral form, $\left\| {{e_i}(t)} \right\| \le \alpha (t) + K \int_0^t {\left\| {{e_i}(\tau )} \right\|{\rm{d}}\tau }$, where $\alpha (t): = \left\| {{e_i}(0)} \right\| + K\int_0^t {\left\| {{e_o}(\tau )} \right\|{\rm{d}}\tau }$, and using Gronwall's
lemma, we obtain
$$
\left\| {{e_i}(t)} \right\| \le \alpha (t) + \int_0^t {\alpha (s)K\exp \left( {\int_s^t {Kd\tau } } \right){\rm{ds}}}.
$$
By definition, ${\left\| {{e_o}} \right\|_ \infty} \ge \left\|
{{e_o}(t)} \right\|$ for any $t \in J_0$. It follows that $\alpha (t) \le \left\| {{e_i}(0)} \right\| + Kt\left\| {{e_o}} \right\|_\infty$. Simplifying the integral of the right-hand side of the above inequality gives
$$
 {\left\| {{e_i}}(t) \right\|} \le (e^{Ka/2}-1){\left\| {{e_o}} \right\|_\infty}+  e^{Ka/2} \left\| {{e_i}(0)} \right\|,
$$
for any $t\in J_0$.
Combining the above inequality with $\left\| {e} \right\|_\infty
\le \left\| {e_i } \right\|_\infty + \left\| {e_o }
\right\|_\infty$, one can obtain  (\ref{eeo}).
 \quad
\end{proof}
\medskip

\textbf{Remark:} The above lemma provides a bound for $\|e_i(t)\|$ in terms of $\|e_o\|_\infty$ and $\|e_i(0)\|$.
We have $\left\|e_i(0)\right \|=0$ when the
initial condition of the reduced model is given by $\hat x_0=P
x_0$ for  (\ref{rom1}). In this situation, (\ref{eeo})
becomes ${\left\| {{e}} \right\|_\infty} \le e^{K a/2} {\left\|
{{e_o}} \right\|_\infty}$. Considering $\left\| {e}
\right\|_\infty \ge \left\| {e_o} \right\|_\infty$, $\left\| {e}
\right\|_\infty=0 $ holds if and only if $\left\| {e_o}
\right\|_\infty=0 $.

Obviously,  $J_0$  is not the maximal  time interval of the
existence and uniqueness of  $x(t)$ and $\hat x(t)$. For
convenience, we simply assume that $x(t)$ and $\hat x (t)$
globally exist on $J=[0,T]$ throughout the rest of this article.
Otherwise, we can shrink $J$ to a smaller interval, which starts
from $0$, such that both $x(t)$ and $\hat x(t)$ are well-defined
on $J$. Let $\mathcal{D}$ be an open set that contains $x(t)$,
$\tilde x(t)$, and $\hat x(t)$ for all $t\in J$. Under this
assumption, Lemma \ref{bound} is still valid if $J_0$ and
$B_{b}(x_0)$ are substituted by $J$ and $\mathcal{D}$,
respectively.
\medskip

\subsection{POD}
In order to provide an accurate description for the original
system, the POD method can be used to deliver a set of empirical eigenfunctions
such that the error   for representing the given data onto the
spanned subspace is optimal, in a  least squares  sense
\cite{HolmesP:02a}. Assume that $m$ precomputed snapshots
  form a matrix, $X:=[x(t_1),..., x(t_m)]$. Then, the truncated singular value decomposition (SVD)
\begin{equation}
X \approx \Phi \Lambda \Psi^T
\end{equation}
provides  the POD basis matrix $\Phi\in \mathbb{R}^{n\times k}$, where $\Lambda \in \mathbb{R}^{k
\times k}$ is a diagonal matrix that consists of the first $k$ nonnegative singular values arranged in decreasing order.
  $P$ is then obtained by $\Phi \Phi^T$.

Let $E$  denote  the energy of the full system, which is   approximated by  the square of the Frobenius norm of snapshot matrix $X$, $E=\int_{{0}}^{{T}} {{{\left\| {x(t)} \right\|}^2} \: \mathrm{d} t} \approx \|X\|_F^2 =\sum\nolimits_{\alpha = 1}^r {\lambda _\alpha^2}$, where $r=\min(n,m)$. Let $E'$ denote the energy  in the optimal $k$-dimensional subspace, $E'=\int_{{0}}^{{T}} {{{\left\| {{P}x(t)} \right\|}^2} \: \mathrm{d} t}  \approx \|P X\|_F^2= \sum\nolimits_{\alpha = 1}^k {\lambda _\alpha^2} $. A criterion can be set to limit the approximation error in the
energy by a certain fraction $\eta$. Then, we seek $k \ll r$ so
that
\begin{equation}\label{num}
{{E'} \mathord{\left/
 {\vphantom {{E'} E}} \right.
 \kern-\nulldelimiterspace} E} > \eta.
\end{equation}

The key for POD and other projection-based reduced models is to
find a  subspace where all the state
vectors approximately reside.  Although these methods can significantly
increase the computational speed during the online stage,  the
cost of data ensemble construction in the offline stage is often
very expensive.  For these
reasons, developing an inexpensive online  manifold learning
technique is a desirable objective.

\section{SIRM} \label{sec:SIRM}
The SIRM method is introduced by combining subspace iteration with
a model reduction procedure in this section. The idea of subspace
construction is to enhance the POD method by feeding it with
information drawn from the observed state of the system and its
time derivation. Then, a more precise solution  is solved by
projecting the original system onto this subspace. The subspace
construction is carried out iteratively until a convergent
solution is achieved.
\subsection{Algorithm of SIRM}
 In this article, a $k$-dimensional
subspace $S$ is called \emph{invariant} of $x(t)$ (or invariant
for short) if $x(t) \in S$ for all $t\in J$. In this case,
$P x(t) = x(t)$,
  which means that $P$ is an invariant projection operator on the trajectory and that $\tilde x(t)=x(t)$.  As mentioned above, $e_o(t)=0$
  holds if and only if  $e(t)=0$. Then, $\hat x(t)=x(t)$.
Inserting (\ref{general}) and (\ref{rom1}) into $\dot{\hat
x}(t)=\dot x(t)$, one can achieve
$  P f(t,x) = f(t,x)$,
which is equivalent to  $f(t, x) \in S$. In fact, $(x(t),f(t,x))$
can be considered  a point in the tangent bundle $TS$, which
coincides with the Cartesian product of $S$ with itself. As an invariant projection, $P$ preserves
not only the state vectors along the solution orbit but also
 the associated tangent vectors, i.e., the dynamics.

On the $j$th iteration,  the aim is to construct a subspace $S^j$ such
that both  $\hat x^{j-1}$ and $f(t, \hat x^{j-1})$ are invariant
under the associated projection operator $P^j$, i.e.,
\begin{equation} \label{identityjj}
P^j (\hat x^{j-1}) = \hat x^{j-1} ,
\end{equation}
\begin{equation} \label{identityj1}
 P^j (f(t,\hat x^{j-1})) = f(t, \hat x^{j-1}).
\end{equation}
Thus, both $\hat x^{j-1} (t)$ and $f(t,\hat x^{j-1} (t))$ reside
in $S^j$ for all $t\in J=[0,T]$.

If the solution orbit is given at discrete times $t_1, ..., t_m$,
then we have an $n\times m$ state matrix
\begin{equation}
 \hat X^{j}:=[\hat x^{j}(t_1),...,\hat x^{j}(t_m)].
\end{equation}
 Accordingly, the samples of tangent vectors along the approximating orbit can form another $n \times m$
matrix,
\begin{equation} \label{orbit}
 \hat F^{j}:=[f(t_1, \hat x^{j}(t_1)),...,f(t_m, \hat x^{j}(t_m))].
\end{equation}
 A combination of $\hat
X^j$ and $\hat F^j$ gives an \emph{information matrix}, which is
used to represent  an extended data ensemble
\begin{equation} \label{infomatrix}
\hat Y^j:=[\hat X^j, \gamma  \hat F^j],
\end{equation}
where $\gamma$ is a weighting coefficient.  The basis vectors of
$S^j$ can be obtained by using  SVD of $\hat Y^{j-1}$.
$\gamma=1$ is a typical value that is used to balance the
truncation error of $\hat X^j$ and $\hat F^j$.  It can be noted
that a large  $m$ value will lead to intensive computation, but
the selected snapshots should reflect the main dimensions of
states and tangent vectors along the solution trajectory. When the
width of each time subinterval (partitioned by $t_i$) approaches
zero, $S^j$ can be given by the column space of  $\hat Y^{j-1}$.

\begin{algorithm}
\caption{SIRM}
\label{alg:SIRM}
\begin{algorithmic}
\REQUIRE The initial value problem (\ref{general}). \ENSURE An
approximate solution $\hat x(t)$. \STATE Set a test function $\hat
x^0(t)$ as the trial solution. Initialize the iteration number
$j=0$.
 \REPEAT
 \STATE 1: Update the iteration number $j=j+1$.
 \STATE 2: Assemble snapshots of an approximate solution $\hat x^{j-1}(t)$ into
  matrix form $\hat X^{j-1}$.
 \STATE 3: Compute vector
 field matrix $\hat F^{j-1}$ associated with snapshots in $\hat X^{j-1}$.
 \STATE 4: Form an information matrix for the extended data ensemble  $\hat Y^{j-1}=[\hat X^{j-1}, \gamma \hat
 F^{j-1}]$.
 \STATE 5: Based on $\hat Y^{j-1}$, compute the empirical eigenfunctions $\Phi^j$ through POD.
 \STATE 6: Project the original equation onto  a linear subspace spanned by
 $\Phi^j$ and form a reduced model.
 \STATE 7: Solve the reduced model and obtain an approximate solution $z^j(t)$ in the
subspace coordinate system.
 \STATE 8: Express the updated solution in the original coordinate system $\hat x^j(t)=\Phi^j z^j(t)$.
  \UNTIL{$\left \| \hat x^{j}- \hat x^{j-1} \right \|_{\infty}<
    \epsilon$, where $\epsilon$ is the error tolerance.
    \STATE Obtain the final approximate solution $\hat x(t)=\hat x^j(t)$}.
\end{algorithmic}
\end{algorithm}
Algorithm \ref{alg:SIRM} lists the comprehensive procedures of the
SIRM method. A new subspace $S^j$ is constructed in each
iteration, followed by an approximate solution $\hat x^{j-1}(t)$.
As $\hat x^j(t) \to x(t)$, $S^j$ approaches an invariant subspace.
For this reason, SIRM is an iterative manifold learning procedure,
which approximates an invariant subspace by a sequence of
subspaces. A complete iteration cycle  begins with a collection of
snapshots from the previous iteration (or an initial test
function).  Then,  a subspace spanned by an information matrix is
constructed. Empirical eigenfunctions are generated by POD, and
finally  a reduced-order equation obtained by Galerkin projection
(\ref{rom}) is solved.

\subsection{Convergence Analysis} \label{sec:analysis}
In this subsection, we first provide  a local error bound for the
sequence of approximate solutions $\{\hat x^j(t)\}$ obtained by
SIRM, which paves the way for the proof of local  and global
convergence of the sequence.

 It can be noted that  both $x^{j-1}(t) \in S^j$ and $f(t,x^{j-1})\in  S^j$
hold for all $t\in J_0$ only in an ideal situation. If $S^j$ is
formed by extracting the first few dominant modes from the
information matrix of the extended data ensemble
(\ref{infomatrix}), neither (\ref{identityjj}) nor
(\ref{identityj1}) can be exactly satisfied.
Let $\varepsilon^j$  quantify the projection error,
\begin{equation} \label{epi}
\varepsilon^j : =  \int_{{0}}^{a/2} {{{\left\| {\left( {I - {P^j}}
\right) {{\hat x}^{j - 1}}(\tau)} \right\|}^2} \: \mathrm{d} \tau
+ } \gamma^2 \int_{{0}}^{a/2} {{{\left\|
 {\left( {I - {P^j}} \right)f(\tau ,{{\hat x}^{j - 1}})} \right\|}^2} \: \mathrm{d} \tau
}.
\end{equation}
If  SVD is used to construct the empirical eigenfunctions,
$\varepsilon^j$ can be estimated by
\begin{equation}
\varepsilon^j  \approx \sum\limits_{\alpha = k^j+1}^{r} {(\lambda
_\alpha^j)^2},
\end{equation}
where  $\lambda_\alpha^j$ is the $\alpha$th singular value of the
information matrix $\hat Y^{j-1}$, and $k^j$ is the adaptive
dimension of $S^j$  such that the truncation error  produced by SVD is
bounded by $\varepsilon^j \le \varepsilon$. The following lemma
gives an error bound for the limit of the sequence $\{\hat
x^j(t)\}$.

\medskip
\begin{lemma}\label{proposition}
Consider solving the initial value problem \emph{(\ref{general})}
over the interval $J_0=[0,a/2]$ by the SIRM method.  $a$, $J_a$,
$b$, $B_{b}(x_0)$, $M$, $P$, $x(t)$, $\tilde x(t)$,  $\hat x(t)$,
$e(t)$, $e_o(t)$, and $e_i(t)$ are defined as above. The
superscript $j$ denotes the $j$th iteration. Suppose $f(t,x)$ is a
uniformly Lipschitz function of $x$ with constant $K$ and a
continuous function of $t$  for all $(t, x) \in J_a \times
B_b(x_0)$.  Then $\{\hat x^j(t)\}$   approaches  $x(t)$ with an
upper bound of $\left\|e\right\|_\infty$ given by
\begin{equation} \label{chi}
\chi  = \frac{{\sqrt{a\varepsilon} {e^{Ka/2}}}}{{{\sqrt{2}\gamma
}(1 - Ka{e^{Ka/2}}/2)}} + \frac{2\theta e^{Ka/2}}{{1 -
Ka{e^{Ka/2}}/2}}
\end{equation}
for all $t \in J_0$ provided that
\begin{equation}\label{limita}
a<\min\left [\frac {b}{M},  \frac{2e^{-Kb/2M}}  {K} \right ],
\end{equation}
where $\theta$ is the maximal error of the initial states in
reduced models, and $\theta<b/2$.
\end{lemma}
\medskip

\begin{proof}
As proved in Lemma \ref{bound},  $x(t)$, $\tilde x(t)$, and $\hat
x(t)$ are well-defined over the interval $J_0$. Moreover, $x(t)$,
$\tilde x(t)$, and $\hat x(t)\in B_b(x_0)$ for all $t\in J_0$.  Multiplying (\ref{ee}) on the left by
$I-P^j$, we obtain the
   evolution equation for $e_o^j(t)$,
$$
 \dot e^j_o = -(I-P^j)f(t,x),
$$
which is equivalent to
\begin{equation}\label{eo}
\dot e_o^j = (I - {P^j})[f(t,x + {e^{j - 1}}) - f(t,x)] - (I -
{P^j})f(t,x + {e^{j - 1}}).
\end{equation}
Considering that $x(t) \in B_b(x_0)$, $\hat x^{j-1}(t) \in
B_b(x_0)$ for all $t\in J_0$ and $f(t,x)$  is a uniformly
Lipschitz function for all $(t,x)\in J_a \times B_b(x_0)$ with
constant $K$, it follows that
\begin{equation}\label{lipK}
 \left\| {f(t,x + e^{j - 1} ) - f(t,x)} \right\| \le K \left\|
{e^{j - 1}(t) } \right\|.
\end{equation}
Since $P^j$ is a projection matrix, we have
$\left\|I-P^j\right\|=1$. This equation together with (\ref{eo})
and (\ref{lipK}) yields
\begin{equation}\label{eobound}
\left\| \dot e_o^j (t)\right \| \le K\left\| {{e^{j - 1}}(t)}
\right\| + \left\| {(I - {P^j})f(t,x + {e^{j - 1}})} \right\|.
\end{equation}
For $h>0$, the expansion of $e_o^j(t+h)$ gives
\begin{equation}\label{taylor}
\left\| {e_o^j (t + h)} \right\| \le \left\| {e_o^j (t)} \right\|
+ h\left\| {\dot e_o^j (t)}  \right\| + \mathcal{O}(h^2 ).
\end{equation}
Rearranging (\ref{taylor}) and applying (\ref{eobound}) results in
\begin{equation}\label{diffbound1}
\frac{{\left\| {e_o^j(t + h)} \right\| - \left\| {e_o^j(t)}
\right\|}}{h} \le K\left\| {{e^{j - 1}}(t)} \right\| + \left\| {(I
- {P^j})f(t,x + {e^{j - 1}})} \right\| + \mathcal{O}(h),
\end{equation}
 where the $\mathcal{O}(h)$ term may be uniformly bounded
independent of $e_o^{j}(t)$. Integrating (\ref{diffbound1}) with
respect to $t$ yields
\begin{equation}\label{e0e1}
\left\| {e_o^j(t)} \right\| \le K\int_{{0}}^t {\left\| {{e^{j -
1}}(\tau )} \right\| \: \mathrm{d} \tau}   + \int_{{0}}^t {\left\|
{(I - {P^j})f(\tau ,x + {e^{j - 1}})} \right\| \: \mathrm{d} \tau
}+\left\|e_o^j(0)\right\|.
\end{equation}
For  $t\in J_0$, the first term on the right-hand side is bounded
by $Ka\left\| {{e^{j - 1}}} \right\|_\infty$/2. Using the definition of $\varepsilon^j$ in (\ref{epi}) and the fact that
$\varepsilon^j \le \varepsilon$ for each $j$, we obtain
\begin{equation*}
\varepsilon \ge \gamma^2 \int_{{0}}^{a/2} {{{\left\|
 {\left( {I - {P^j}} \right)f(\tau ,{{\hat x}^{j - 1}})} \right\|}^2} \: \mathrm{d} \tau
}.
\end{equation*}
By the Cauchy--Schwarz inequality, the second term on the right-hand side of (\ref{e0e1}) is  bounded by $\sqrt{a\varepsilon
/2{\gamma ^2}}$ when $t\le a/2$. It follows that
$$\left\| {e_o^j(t)} \right\| \le
\frac{{Ka{{\left\| {{e^{j - 1}}} \right\|}_{\infty}}}}{2} +
\sqrt{\frac{{a\varepsilon }}{{2{\gamma ^2}}}} + \left\| {e_o^j(0)}
\right\|.$$ Using (\ref{eeo}) in Lemma \ref{bound}, this
inequality yields
\begin{equation}\label{finalbound}
{\left\| {{e^j}} \right\|_\infty} \le \frac{Ka { {{ {{e^{Ka/2}} }
}}} {\left\| {{e^{j - 1}}} \right\|_\infty}}{2} +
\frac{{\sqrt{a\varepsilon} e^{Ka/2} }}{{\sqrt{2}{\gamma }}}
+e^{Ka/2}\left\|e_o^j(0)\right\|+e^{Ka/2}\left\|e_i^j(0)\right\|.
\end{equation}
If the error of the initial condition is bounded by
$\left\|e^j(0)\right\| \le \theta$ for each iteration, then $
\left\|e_o^j(0)\right\|\le \theta$, and
$\left\|e_i^j(0)\right\|\le \theta$. As a result,
\begin{equation}\label{finalbound1}
{\left\| {{e^j}} \right\|_\infty} \le \frac {Ka { {{ {{e^{Ka/2}} }
}}} {\left\| {{e^{j - 1}}} \right\|_\infty}}{2} +
\frac{{\sqrt{a\varepsilon} e^{Ka/2} }}{{\sqrt{2}{\gamma }}}
+2\theta e^{Ka/2}.
\end{equation}
By (\ref{limita}), $ Ka { {{ {{e^{Ka/2}} } }}}/2 < Ka { {{
{{e^{Kb/2M}} } }}}/2  < 1$.  Using the definition of $\chi$ in
(\ref{chi}), (\ref{finalbound1}) can be rewritten as
\begin{equation}\label{eqbound}
\|e^j\|_\infty-\chi \le Ka { {{ {{e^{Ka/2}} }
}}}/2(\|e^{j-1}\|_\infty-\chi).
\end{equation}
It follows that if $\|e^j\|_\infty-\chi>0$  for all
$j$, it converges to $0$ linearly. Otherwise, once
$\|e^{j_0-1}\|_\infty-\chi\le0$ for some $j_0$, then
$\|e^{j_0}\|_\infty\le \chi$, and so does $\|e^j\|_\infty$ for all
$j>j_0$. Therefore, we have
\begin{equation}
\mathop {\limsup }\limits_{j \to +\infty} {\left\| {{e^j}}
\right\|_{\infty}} \le \chi,
\end{equation}
 which means $\left \| e^j(t) \right \|$ is
bounded by $\chi $ as $j \to +\infty$ for all $t\in J_0$. \quad
\end{proof}
\medskip

The first term  of $\chi$ is introduced by the truncation error.
By decreasing the width of time intervals among neighboring
snapshots and increasing the number of POD modes, we can limit the
value of $\varepsilon$. The second term of $\chi$ is the magnified
error caused by $e^j(0)$. If both $\chi$ and $e^j(0)$ approach 0,
we have the following theorem.
\medskip
\begin{theorem} [{\rm  local convergence of SIRM}]\label{maintheorem}
Consider solving the initial value problem \emph{(\ref{general})}
over the interval $J_0=[0,a/2]$ by the SIRM method.  $a$, $J_a$,
$b$, $B_{b}(x_0)$, $M$, $P$, $x(t)$, $\tilde x(t)$,  $\hat x(t)$,
$e(t)$, $e_o(t)$, and $e_i(t)$ are defined as above. The superscript
$j$ denotes the $j$th iteration. Suppose $f(t,x)$ is a uniformly
Lipschitz function of $x$ with constant $K$ and a continuous
function of $t$  for all $(t, x) \in J_0 \times B_b(x_0)$. For
each iteration, the reduced subspace $S^j$ contains $x_0$ and the
initial state for the reduced model is given by $\hat x^j_0=P^j
x_0$.   Moreover, (\ref{identityj1}) is satisfied. Then the
sequence $\{\hat x^j(t)\}$ uniformly converges to $x(t)$ for all
$t \in J_0$, provided that
\begin{equation}\label{bounda}
a<\min\left [\frac {b}{M},  \frac{2e^{-Kb/2M}}  {K } \right ].
\end{equation}
\end{theorem}
\medskip
\begin{proof}
Since the initial state $\hat x^j_0$ is the projection of $x_0$
onto $S^j$, we have $e_i^j(0)=0$. Meanwhile, $x_0\in S^j$ results
in $e_o^j(0)=0$. Then the initial error satisfies $e^j(0) =0$. On
the other hand, (\ref{identityj1}) requires that $f(t,\hat
x^{j-1})$ is invariant under the projection operator $P^j$, i.e.,
 $ (I-P^j) f(t,\hat x^{j-1}) = 0$, which leads to
$\varepsilon^j=0$. Therefore, in Lemma \ref{proposition}, both
$\varepsilon$ and $\theta$ approach  0, and so does $\chi$. As a
consequence, $\{\hat x^j(t)\}$ converges to the fixed point $x(t)$
for all $t\in J_0$. \quad
\end{proof}
\medskip

It can be noted that the error bound $\chi$ of the SIRM method is
completely determined by $\theta$ and $\varepsilon$. As an
alternative to (\ref{infomatrix}), a more straightforward form of
the information matrix for the extended data ensemble can be
written as
 \begin{equation}\label{newinfo}
 \tilde Y^j:=[x_0, \gamma \hat F^j],
 \end{equation}
and the SIRM method can still converge to $x(t)$ by Theorem
\ref{maintheorem}. However,  as a Picard-type iteration, SIRM can
only be guaranteed to reduce local error within one iteration. If
$\varepsilon$ and $\theta$ approach  0, (\ref{finalbound1}) can be
rewritten as
 \begin{equation}\label{ejej}
\frac{{{{\left\| {{e^j}} \right\|}_{\infty}}}}{{{{\left\| {{e^{j -
1}}} \right\|}_{\infty}}}} \le \frac{Ka{e^{Ka/2}}}{2}.
 \end{equation}
  When the interval $J_0$ is large, for example, $a>2/K$,  the left-hand side might be greater than $1$. Thus,
although $\hat x^j(t)$ has less local error than $\hat
x^{j-1}(x)$, it might be less accurate in a global sense.

On the other hand, if the information matrix (\ref{infomatrix}) is
applied to the SIRM method, $\hat x^{j-1}(t) \in S^j$ is satisfied
for each iteration. For any $t$, $\left\| {e_o^j (t)} \right\|$
denotes the distance from $x(t)$ to $S^j$, while $\left\| e^{j-1}
(t)\right \|$ denotes the distance from $x(t)$ to $\hat
x^{j-1}(t)$. Recognizing that $\hat x^{j-1}(t)\in S^j$, we have
\begin{equation}
\left\| {e_o^j (t)} \right\| \le  \left\| e^{j-1} (t)\right \|.
\end{equation}
If $\theta$ approaches 0, so does $\|e_i^j(0)\|$. Using
(\ref{eeo}) in Lemma \ref{bound}, one obtains
 \begin{equation}\label{ejej1}
\frac{{{{\left\| {{e^j}} \right\|}_{\infty}}}}{{{{\left\| {{e^{j -
1}}} \right\|}_{\infty}}}} \le {e^{Ka/2}}.
 \end{equation}
This inequality still cannot guarantee that  $\left\| e^{j-1} (t)
\right\| < \left\| e^j (t)\right\|$ for all $t\in J$. However,
when $a>2/K$ it provides a stronger bound than (\ref{ejej}) does,
which can effectively reduce the global error.

So far, we have proved  convergence of SIRM for a local time
interval $J_0=[0, a/2]$. Since the estimates used to obtain $J_0$
are certainly not optimal, the true convergence time interval is
usually much larger. Supposing $J_0 \subset J$, we will next
prove that the convergence region $J_0$ can be extended to $J$
under certain conditions.
\medskip

\begin{theorem}[{\rm  global convergence of SIRM}]\label{globalconvergence}
Consider solving the initial value problem \emph{(\ref{general})}
over the interval $J=[0,T]$ by the SIRM method.   $P$, $x(t)$,
$\tilde x(t)$,  $\hat x(t)$, $e(t)$, $e_o(t)$, and $e_i(t)$ are
defined as above. The superscript $j$ denotes the $j$th iteration.
Suppose $f(t,x)$ is a locally Lipschitz function of $x$ and a
continuous function of $t$  for all $(t, x) \in J\times
\mathcal{D'}$, where $\mathcal{D'}$ is an open set that contains
$x(t)$ for all $t\in J$.  For each iteration, the reduced subspace
$S^j$ contains $x_0$ and the initial state for the reduced model
is given by $\hat x^j_0=P^j x_0$.   Moreover,  (\ref{identityj1})
is satisfied. The sequence $\{\hat x^j(t)\}$ then uniformly
converges to $x(t)$ for all $t \in J$.
\end{theorem}
\medskip

\begin{proof}
 Since $\mathcal{D'}$ is  open, there exists a constant $b$ such that
 $b>0$, and $\mathcal{E}:=  \cup_t {\bar B_b}(x(t)) \subset \mathcal{D'}$. Since $f(t,x)$
is locally Lipschitz on $J\times \mathcal{D'}$ and $\mathcal{E}$
is compact, $f(t,x)$ is Lipschitz on $J \times \mathcal{E}$. Let
$K$ denote the Lipschitz constant for $(t,x)\in J \times
\mathcal{E}$. In addition, we can choose the value of $a$, which
is bounded by (\ref{bounda}). Let $J_m$ be the maximal interval in
$J$ such that for all $t\in J_m$,
 $\hat x^j(t) \to x(t)$
uniformly as $j\to \infty$. Theorem \ref{maintheorem} indicates
that SIRM will generate a sequence of  functions $\{\hat x^j(t)\}$
that uniformly converges to $x(t)$ for all $t\in J_0=[0, a/2]$.
For this reason, we have $J_0 \subset J_m$.

 Now assume $J_m\ne J$. Then, there exists a
$t_i \in J_m$  such that  $t_i+a/2\le T$, but $t_i+a/2 \notin
J_m$. $t_i \in J_m$ means for every $\kappa>0$ there exists an
integer $M_1(\kappa)>0$ such that for all $j$ with
$j>M_1(\kappa)$,  $\hat x^j(t_i)$ uniquely exists and $\left\|\hat
x^j(t_i) - x(t_i)\right\|<\kappa$.

Consider the initial value problem
\begin{equation} \label{eqy}
\dot y=f(t,y); \ \ \ \ \ y(0)=y_0=x(t_i).
\end{equation}
The corresponding reduced model of SIRM at iteration $l$ is given
by
\begin{equation}\label{odeuni}
\dot {\hat y}^l = P^l f(t,  \hat y^{l}); \ \ \ \ \ \hat
y^l(0)=y_0+e^l(t_i),
\end{equation}
where $e^l(t_i)=\hat x^l(t_i)-x(t_i)$.
For an arbitrary small positive number $\kappa'$,
 Lemma \ref{proposition}
implies that    there exists a
positive integer $M_2(\kappa')$ such that whenever
$l>M_2(\kappa')$ and $t\in J_0=[0, a/2]$,  $$\left\| \hat
y^l(t)-y(t)\right\|<\chi+\kappa'.$$ Plugging  $\chi$ from
(\ref{chi}) into this inequality and replacing $\theta$ by $\kappa$, we have
\begin{equation}\label{constraint}
\left\| \hat y^l(t)-y(t)\right\|\le \frac{{\sqrt{a\varepsilon}
{e^{Ka/2}}}}{{{\sqrt{2}\gamma}(1 - Ka{e^{Ka/2}}/2)}} +
\frac{2\kappa e^{Ka/2}}{{1 - Ka{e^{Ka/2}}/2}} +\kappa'.
\end{equation}
In an ideal case, the truncation error is
zero, i.e., $\varepsilon=0$. Then, the right-hand side of
(\ref{constraint}) can be arbitrarily small. The uniqueness lemma for ODE's (Lemma \ref{existence})
yields
$y(t)=x(t_i+t)$ and $\hat y^j(t)=\hat x^j(t_i+t)$. Therefore, for
every $\epsilon>0$, there exists an integer
\begin{equation}
N(\epsilon)=M_1\left( {\frac{{(1 - Ka{e^{Ka/2}}/2)\epsilon }}{{4e^{Ka/2}}}} \right) + M_2\left( {\frac{\epsilon }{2}} \right)
\end{equation}
 such that, as long as $j>N(\epsilon)$,
$\left\| \hat x^j(t)-x(t)\right\|\le \epsilon$ holds for all
$t\in[0, t_i+a/2]$. Moreover, $t_i+a/2\le T$. However, this
contradicts our assumption that $t_i+a/2\notin J_m$.
Therefore, $J_m=J$, i.e., $\hat x^j(t)$ uniformly converges to
$x^t(t)$ for all $t\in J$. \quad
\end{proof}

\subsection{Computational complexity}\label{sec:complexity}
The computational complexity of the SIRM method for solving an
initial value problem is discussed in this subsection.  We follow
\cite{PetzoldLR:03a} when we estimate the computational cost of the procedures related to the standard
POD-Galerkin approach.

Let $\gamma(n)$ be the cost of computing
the original vector field $f(t,x)$, and let $\hat \gamma(k,n)$ be the
cost of computing the reduced vector field $\Phi^Tf(t,\Phi z)$
based on the POD-Galerkin approach.  In the full model,  the
cost of one-step time integration using Newton iteration is given
by $b\gamma(n)/5+b^2n/20$ if all $n\times n$
matrices are assumed to be banded and have $b+1$ entries around the
diagonal~\cite{PetzoldLR:03a}. Thus, the total cost of the full model for
$N_T$ steps is given by
\begin{equation}\label{costfull}
N_T\cdot\left(b\gamma(n)/5+b^2n/20\right).
\end{equation}

Next, we analyze the complexity of Algorithm \ref{alg:SIRM}. Assuming a trial solution is given initially,  the computational cost  for each iteration mainly
involves the following procedures. In procedure 3,  an $n\times m$
vector field matrix, $\hat F^{j-1}$, is computed based on $m$
snapshots. In procedure 5,  from an $n\times 2m$ information
matrix, $\hat Y^{j-1}$,  the empirical eigenfunctions $\Phi^j$ can
be obtained in $4m^2n$ operations by  SVD~\cite{TrefethenLN:97a}.
In procedure 6,   the original system is projected onto a subspace
spanned by $\Phi^j$,  and this cost is denoted by $\beta(k,n)$. For a
linear time-invariant system, $\dot x=Ax$, $\beta(k,n)$ represents
the cost to compute $(\Phi^j)^T A \Phi^j$, which is given by $bnk$
for sparse $A$. For a general system, $\beta(k,n)$ is a nonlinear function of $n$. In procedure 7, the reduced model is evolved for
$N_T$ steps by an implicit scheme to obtain $z^j(t)$. If the reduced model inherits the same scheme from the full model, then one-step time integration needs
$k \hat \gamma (k,n)/5+k^3/15$ operations~\cite{PetzoldLR:03a}. In
procedure 8, an $n\times m$ snapshot matrix $\hat X^j(t)$ is
constructed through $\hat x(t_i)=\Phi z(t_i)$. Table
\ref{tab:complexity} shows the asymptotic complexity for each
basic procedure mentioned above. Let $N_I$ denote
the number of iterations; then, the total cost of Algorithm
\ref{alg:SIRM} is given by
\begin{equation}\label{costrom}
N_I\cdot\left(4m^2n+\beta (k,n)+N_T\cdot\left(k \hat \gamma
(k,n)/5+k^3/15\right)+m\gamma(n)+mnk\right).
\end{equation}

Notice that the first three terms in (\ref{costrom}) represent the
cost for the classic POD-Galerkin method, while construction of
the extended data ensemble needs extra computational overhead,
$m\gamma(n)+mnk$, for each iteration. On one hand, the subspace dimension $k$ is no greater  than the number of sampling points $m$, which means $mnk<4m^2n$.
On the other hand, we can always choose an optimal $m$ value such that $ m\ll N_T$. Thus, the extra computational overhead plays a secondary  role in (\ref{costrom}), and the computational complexity of  Algorithm \ref{alg:SIRM} is approximately  equal to the number of
iterations, $N_I$, multiplied by the cost of the standard POD-Galerkin
approach.

\begin{table}
\begin{center}
\caption{ Complexity of Algorithm \ref{alg:SIRM} for one iteration
using an implicit scheme for time integration}
 \label{tab:complexity}
\begin{tabular}{|l|l|}
\hline
Procedure & Complexity\\
\hline
  Compute a vector field matrix $\hat F^{j-1}$   & $ m\gamma(n)$\\
  \hline
  SVD: empirical eigenfunctions $\Phi^j$    & $4m^2n$\\
  \hline
  Construct a reduced model   &  $\beta(k,n)$\\
  \hline
  Evolve the reduced model  &  $N_T (k \hat \gamma (k,n)/5+k^3/15)$\\
  \hline
Obtain an approximate solution  $\hat X^j$   & $mnk$\\
 \hline
\end{tabular}
\end{center}
\end{table}

 Algorithm \ref{alg:SIRM} does not explicitly specify the
trial solution $\hat x^0(t)$. In fact, the convergence of SIRM
does not depend on $\hat x^0(t)$, as previously shown. Thus, we can simply set $\hat x^0(t)$ as a constant, i.e.,
$\hat x^0(t)=x_0$. However, a ``good'' trial solution could lead
to a convergent solution in fewer iterations and could thus
decrease the total computational cost. For example, if the full
model is obtained by a finite difference method with $n$ grid
points and time step $\delta t$, a trial solution could be
obtained by a coarse model,  using the same scheme but $n/10$ grid
points with time step $10 \times \delta t$. Thus, the coarse model
can cost less than $1\%$ of the required operations in the full model.

Notice that  $\hat \gamma(k,n)
\ll \gamma(n)$ is achieved only when the analytical formula of
$\Phi^T f(t,\Phi z)$ can be significantly simplified, especially
when $f(t,x)$ is a low-degree polynomial  of $x$
\cite{PetzoldLR:03a}. Otherwise, it is entirely possible that the reduced model could be more expensive than the original one.
Because of this effect, there is no guarantee that  Algorithm \ref{alg:SIRM} can speed
up a general nonlinear system.
 However, it should be emphasized that the POD-Galerkin
approach is not the only method that can be used to construct a
reduced model in the framework of SIRM;  in principle, it can be
substituted by a more efficient model reduction technique when
$f(t,x)$ contains a nonlinear term, such as  trajectory piecewise
linear and quadratic approximations~\cite{WhiteJ:03a, WhiteJ:06a,
WeileDS:01, WhiteJ:00a}, the empirical interpolation method
\cite{PateraAT:06a}, or its variant, the discrete empirical
interpolation method~\cite{SorensenDC:10a}. This article, however,
focuses on using
 SIRM to obtain  an accurate solution without a precomputed database. Therefore,
 the   numerical  simulations  in the  next   subsection   are still based on the
 classic POD-Galerkin approach.

\subsection{Numerical Results}\label{sec:numerical} The proposed algorithm,
SIRM, is illustrated in this subsection by  a linear advection-diffusion equation
and a nonlinear Burgers equation. These examples focus on demonstrating the
capability  of SIRM to deliver accurate results using reduced
models. We also show an application of SIRM for the posteriori
error estimation of a coarse model.
\subsubsection{Advection-Diffusion Equation}
Let $u=u(t,x)$. Consider the one-dimensional advection-diffusion
equation with constant moving speed $c$ and diffusion coefficient
$\nu $,
\begin{equation}\label{adv_diff}
u_t=  - cu_x+ \nu u_{xx},
\end{equation}
 on space $x\in[0,1]$. Without loss of generality,
 periodic boundary conditions are applied,
\begin{equation}\label{boundary}
\begin{array}{l}
u(t,0)=u(t,1), \\
 u_x(t,0)=u_x(t, 1). \\
 \end{array}
\end{equation}
 The initial condition is provided by a cubic spline function,
\begin{equation}\label{initial}
u(0,x)=\left\{
\begin{array}{ccl}
 \vspace{3pt} 1-\frac{3}{2}s^2+\frac{3}{4}s^3&  \text{if }&  0\le s \le 1, \\
 \frac{1}{4}(2-s)^3&  \text{if }&  1<s\le2 ,\\ 0&  \text{if }&  s>2,
\end{array} \right.
\end{equation}
where $s=10\times|x-1/3|$. The fully resolved model is obtained
through a high-resolution finite difference simulation with
spatial discretization by $n$ equally spaced grid points. The
advection term is discretized by the  first-order upwind
difference scheme with the explicit two-step Adams--Bashforth method for time
integration, while the diffusion term is discretized by the
second-order central difference scheme with the Crank--Nicolson
method for time integration.

\begin{figure}
\begin{center}
\begin{minipage}{0.48\linewidth} \begin{center}
\includegraphics[width=1\linewidth]{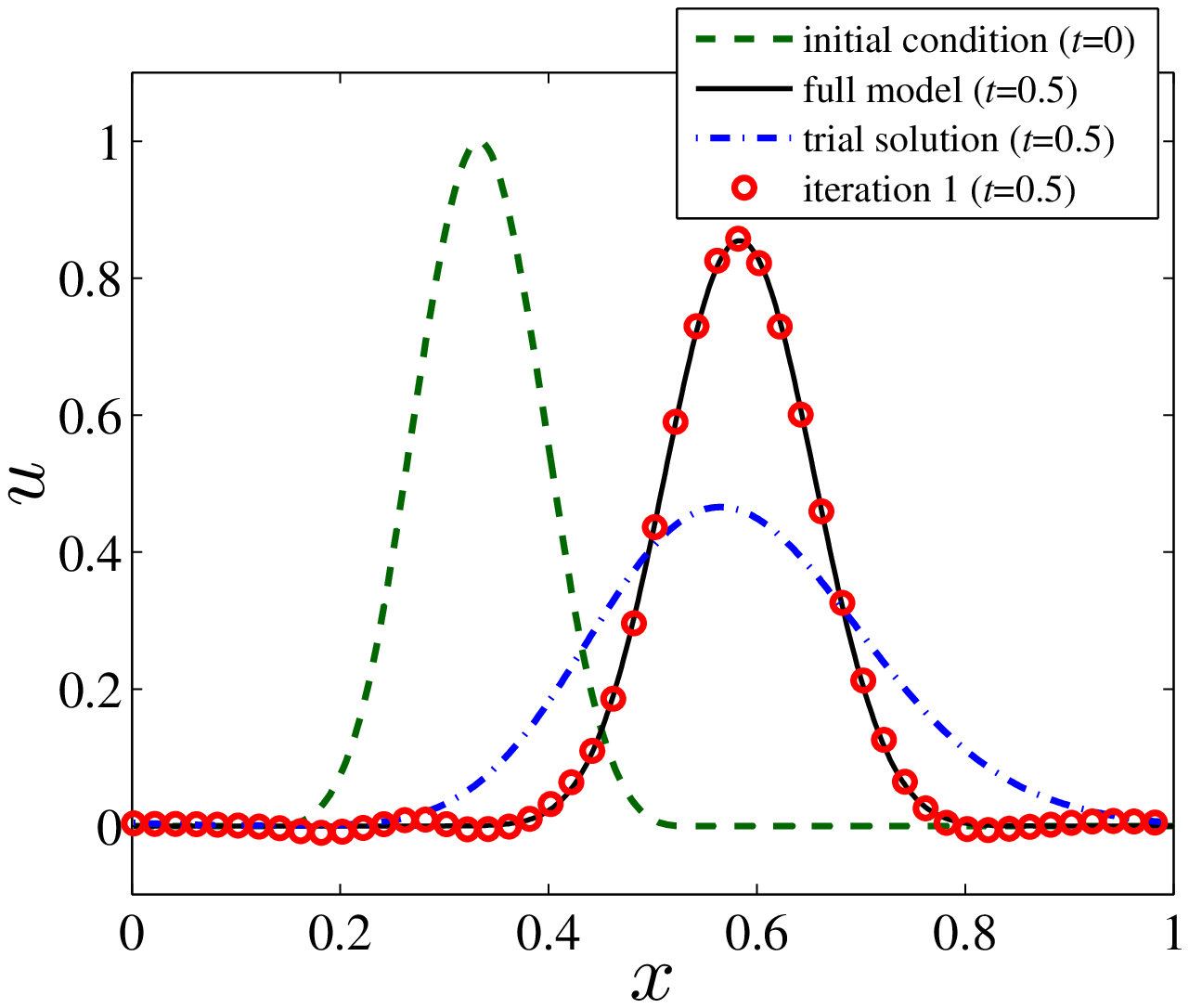}
\end{center} \end{minipage}
\begin{minipage}{0.48\linewidth} \begin{center}
\includegraphics[width=1\linewidth]{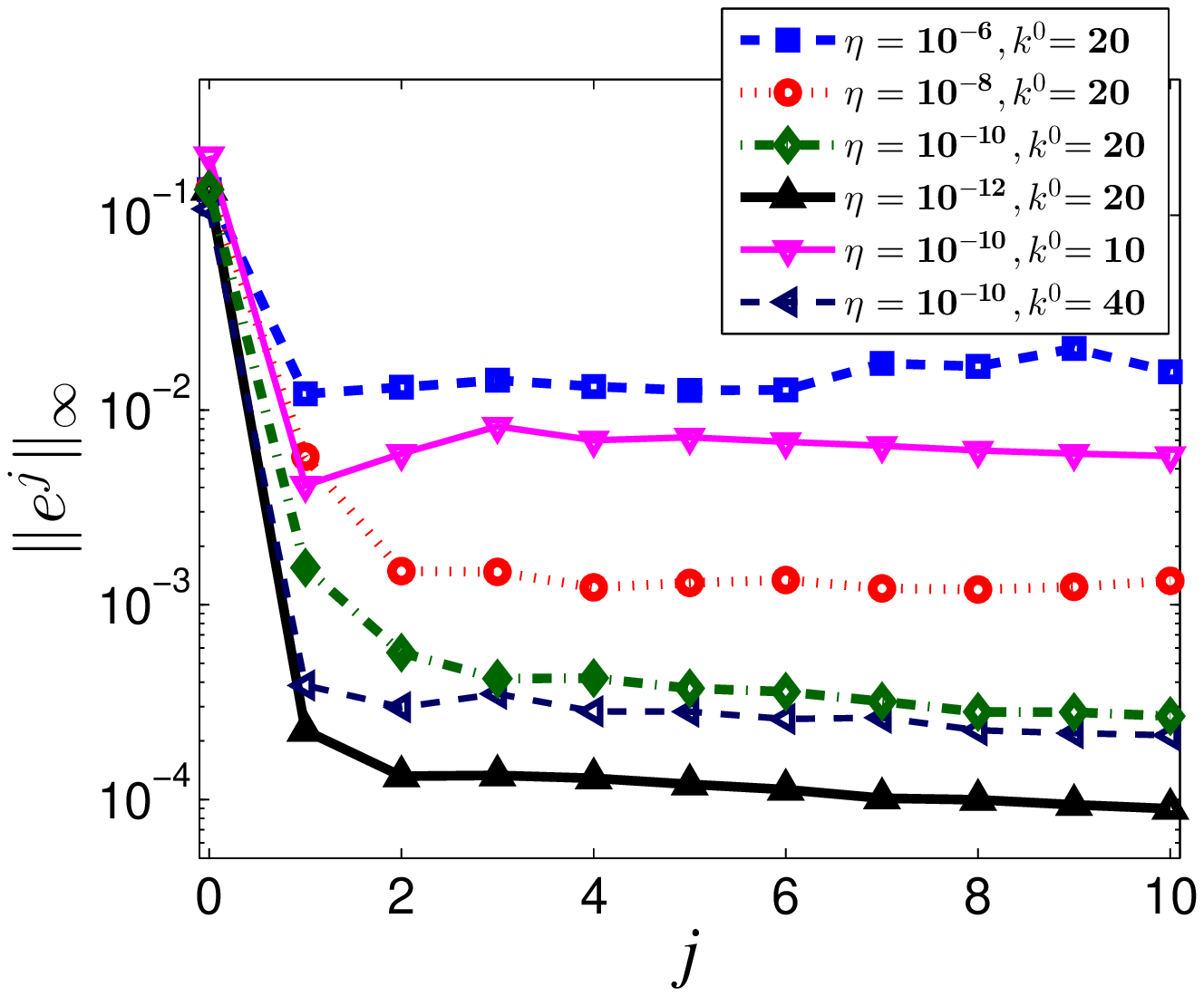}
\end{center} \end{minipage}\\
\begin{minipage}{0.48\linewidth}\begin{center} (a) \end{center}\end{minipage}
\begin{minipage}{0.48\linewidth}\begin{center} (b) \end{center}\end{minipage}
 \caption{ (a) The velocity profiles  at $t=0$ and
$t=0.5$ of the one-dimensional advection-diffusion equation with
constant speed $c=0.5$ and  diffusion coefficient $\nu=10^{-3}$.
$n=500$ grid points are used to obtain the full model for the
fixed space domain $[0, 1]$. The trial solution is obtained by
extracting the first 10 Fourier modes from a coarse model
based on $k^0=20$ grid points. When $\eta=10^{-8}$, it takes  one
iteration for SIRM to obtain an accurate solution by 12 modes.
  (b) Convergence of SIRM for different $\eta$ and $k^0$ values.
Plot of the maximal $L_2$ error, $\|e^j\|_\infty=\sup\{\|\hat
u^j(t)-u(t)\|:t\in[0,0.5]\}$,
 between  the benchmark solution $u(t)$ and the iterative solution $\hat
 u^j(t)$ for $t\in[0,0.5]$.
} \label{fig:advection12}
 \vspace{-3mm}
\end{center}
\end{figure}

 For our numerical experiments, we consider a system with $c=0.5$
and $\nu=10^{-3}$, which gives rise to a system with diffusion as
well as advection propagating to the right. This can be seen in
Figure \ref{fig:advection12}(a), where the initial state and the
final state (at $t = 0.5$) are shown. The full model (reference
benchmark solver) is computed through $n=500$ grid points. Thus, the unit step can be
set as $\delta t=10^{-3}$ such that the Courant--Friedrichs--Lewy (CFL)
condition  is satisfied for  the stability
requirement, i.e., $c\delta t/ \delta x \le 1$. In order to
initialize SIRM, a smaller simulation is carried out by the finite
difference method with  a coarse grid of $k^0=20$  and a larger
time step of $2.5\times 10^{-2}$. In order to obtain a smooth
function for the  trial solution, the coarse solution is
filtered by extracting the first $10$ Fourier modes.
When $\eta=10^{-8}$, the full-order equation is projected
onto a subspace spanned by $k=12$ dominant modes during the first iteration and a better
approximation  is obtained. For different $\eta$ and $k^0$ values, Figure
\ref{fig:advection12}(b) compares the maximal $L_2$ error
 between  the benchmark solution $u(t)$ and the iterative solution $\hat
 u^j(t)$ for $t\in[0,0.5]$ in the first $10$ iterations. Each
 subspace dimension is adaptively determined by (\ref{num}).
 If $k^0=20$, the first three iterations of SIRM
respectively use 9, 9, and 11 dominant modes when $\eta=10^{-6}$;
  use 12, 14, and 14 dominant modes when $\eta=10^{-8}$;  use  14, 17,
and 18 dominant modes when $\eta=10^{-10}$; and  use 17, 19, and
20 dominant modes when $\eta=10^{-12}$. As expected, a smaller
$\eta$ value results in a smaller truncation error  produced by SVD and
the total error, $\|e^j\|_\infty$, for an approximate solution.
Meanwhile, a trial solution with a higher initial dimension, $k^0$,
could also significantly decrease the error for the first 10
iterations. It is also noted that $\|e^j\|_\infty$ is not a
monotonically decreasing function of $j$, especially when $\eta=10^{-6}$
and $k^0=20$. This does not contradict the convergence
analysis in the previous subsection. As a variant of the Picard
iteration, the SIRM method achieves a better local solution in
each iteration. As (\ref{ejej1}) indicates, we can only guarantee
 ${\left\| {{e^j}} \right\|}_{\infty} \le {e^{Ka/2} {\left\| {{e^{j
- 1}}} \right\|}_{\infty}}$ in a global sense.

Before switching to the next numerical example, we compare the performance of  SIRM  with another
online manifold learning technique, DIRM~\cite{PetzoldLR:02a}. The DIRM method splits the whole system into $m_n$ subsystems. Starting with a trial solution, DIRM  simulates each subsystem in turn and repeats this process until a globally convergent solution is obtained. For iteration $j$,
  DIRM connects the unreduced subsystem $i$ with the reduced
  versions   of all other subsystems and simulates the resulting
  system
\begin{equation} \label{eq:SIRM}
\begin{array}{l}
 \dot x_i^j = {f_i}(t,{X^j_i}), \\
 \dot z_l^j = (\Phi _l^j)^T{f_l}(t,{X^j_i}), \ \ \ \ \ \ \ \ \ \ l=1,...,i-1, i+1, ..., m_n, \\
 \end{array}
 \end{equation}
where  $X^j_i=[\Phi_1^j z_1^j; ... ;\Phi_{i-1}^j
z_{i-1}^j; x_i^j; \Phi_{i+1}^j z_{i+1}^j; ... ; \Phi_{m_n}^j
z_{m_n}^j]$. If $x_i^j \in \mathbb{R}^{n/m_n}$ and $z_l^j \in \mathbb{R}^{k}$,  the reduced model of DIRM has a dimension of $m_nk+m/m_n$. Since DIRM reduces the dimension for each subsystem, rather than the original system,  it inevitably keeps some
redundant dimensions.

\begin{table}
\begin{center}
\caption{The minimal subspace dimension of DIRM and SIRM that is
required for solving the one-dimensional advection-diffusion equation when the
error of the first iteration is smaller than $10^{-3}$, i.e.,
$\|e^1\|_\infty<10^{-3}$. Parameter values are $n=500$, $c=0.5$,
$\delta t=10^{-3}$. The time domain is $[0, 0.5]$. The trial
solution is obtained by extracting the first 10 Fourier modes from
a coarse simulation based on $20$ grid points.}
 \label{tab:DIRM}
\begin{tabular}{|c|c|c|c|c|}
\hline
  & $\nu=10^{-1}$ &  $\nu=10^{-2}$ & $\nu=10^{-3}$ & $\nu=10^{-4}$\\
  \hline
  DIRM & 92 & 92 & 116& 116\\
SIRM & 13 & 12 & 15 & 16 \\
 \hline
\end{tabular}
\end{center}
\end{table}

Table \ref{tab:DIRM} compares the minimal subspace dimension of
DIRM and SIRM that is required for solving
(\ref{adv_diff}) when the error of the first iteration is smaller
than $10^{-3}$. We use all the aforementioned parameters except
scanning $\nu$ from $10^{-1}$ to $10^{-4}$. For the DIRM
application, the whole system with $n=500$ is divided into 25
subsystems, and the dimension of each subsystem is 20. When $\nu$
is greater than $10^{-2}$,  the DIRM method uses three modes for each
subsystem, and therefore the dimension of DIRM
is $3\times 24+20=92$. When $\nu$ decreases to $10^{-3}$ and less,  DIRM requires
four modes for each system in order to maintain high accuracy, and
the subspace dimension grows to $4 \times 24 + 20 =116$. Since SIRM requires
 fewer modes and  simulates only one reduced system rather than
20 subsystems,  it is much more efficient than
DIRM for solving  (\ref{adv_diff}).

\subsubsection{Viscous Burgers Equation}
The viscous Burgers equation  is  similar to the
advection-diffusion equation except, in the case of the viscous
Burgers equation, the advection velocity is no longer constant.
The general form of the one-dimensional Burgers equation  is given
by
\begin{equation}\label{burgers}
u_t =  - uu_x + \nu u_{xx},
\end{equation}
where $\nu$ is the diffusion coefficient.  Let $\Omega=[0, 1]$ denote
the computational domain. Periodic boundary conditions
(\ref{boundary}) are applied. The cubic spline function
(\ref{initial}) is used for the initial condition.

\begin{figure}
\begin{center}
\begin{minipage}{0.48\linewidth} \begin{center}
\includegraphics[width=1\linewidth]{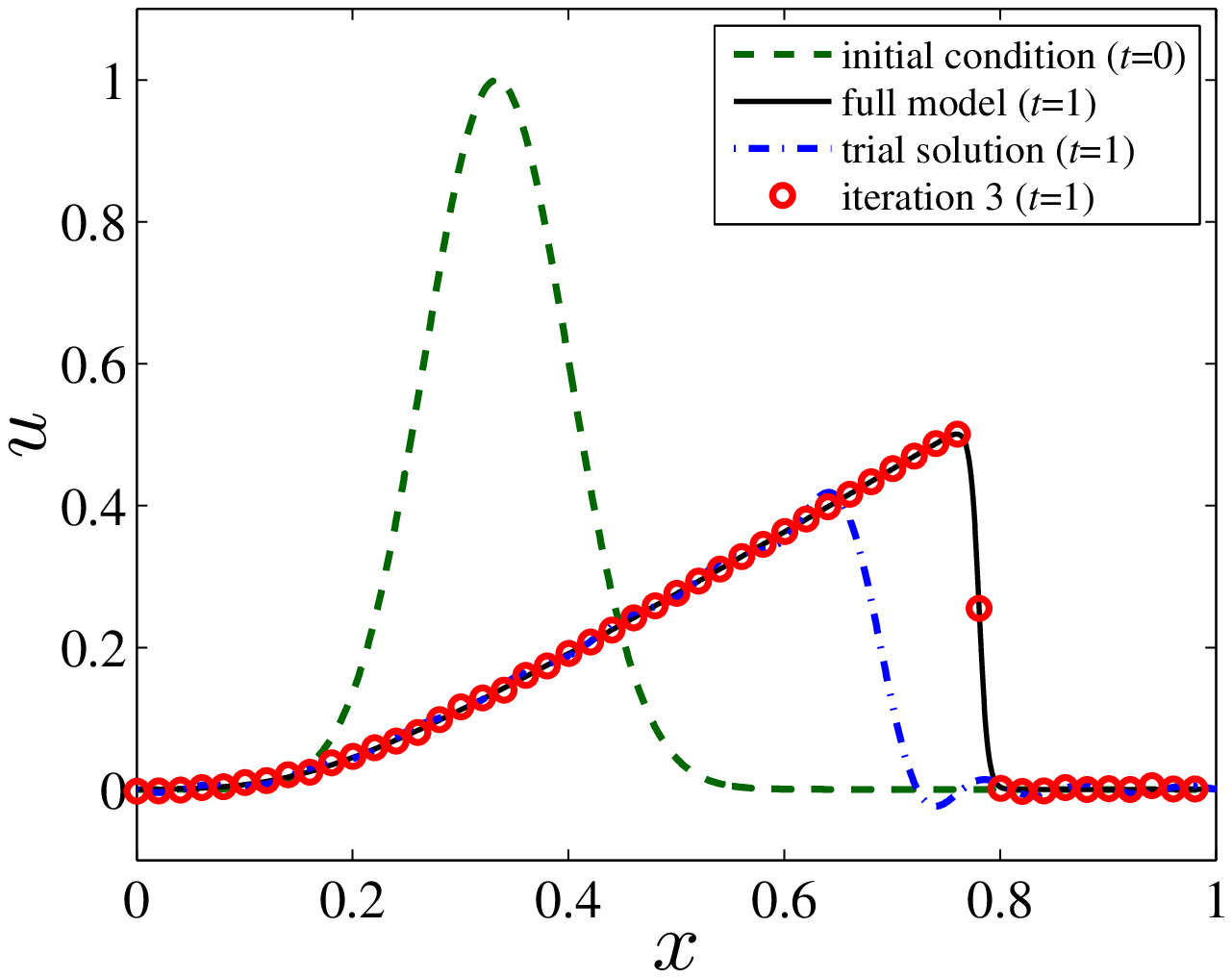}
\end{center} \end{minipage}
\begin{minipage}{0.48\linewidth} \begin{center}
\includegraphics[width=1\linewidth]{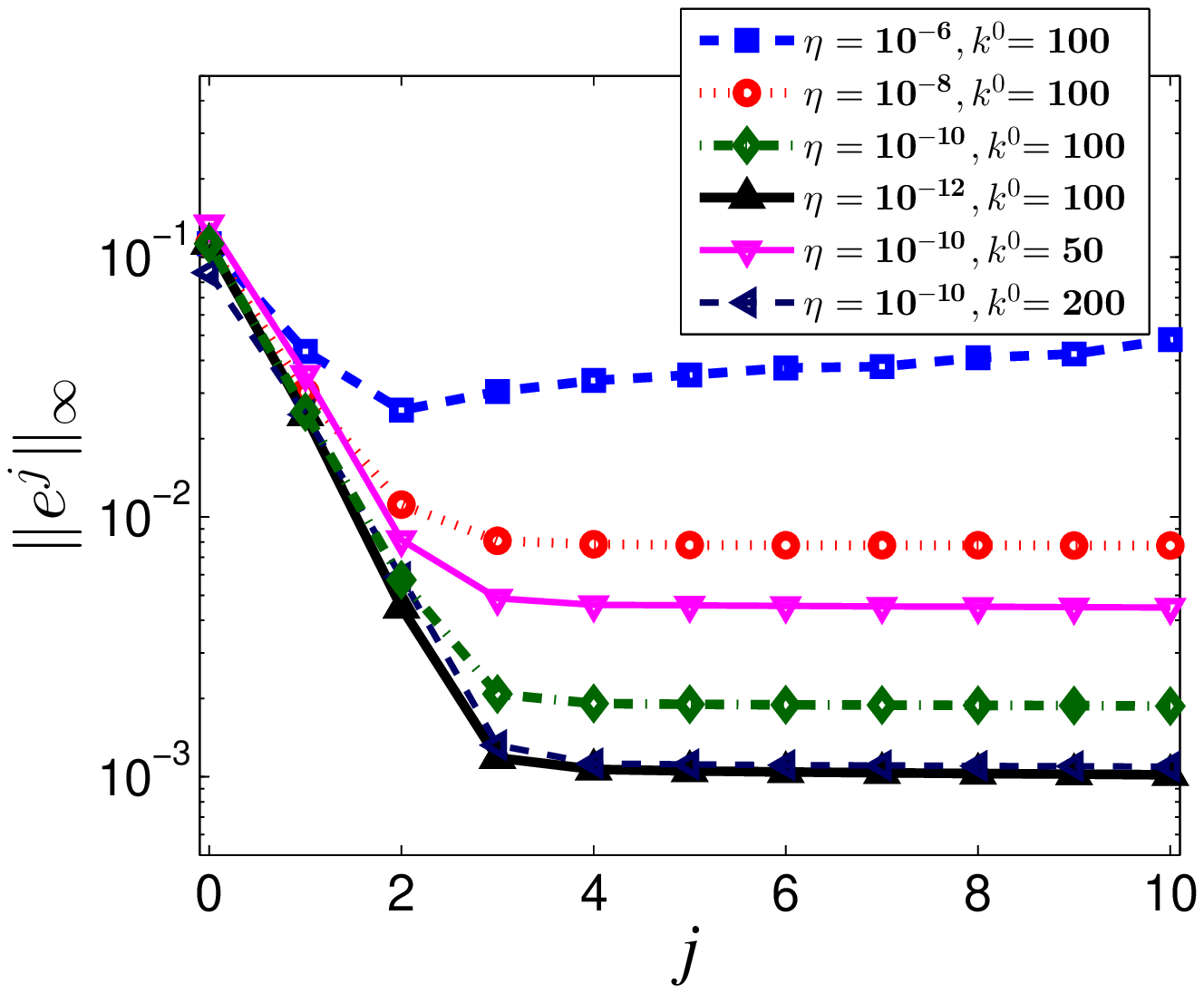}
\end{center} \end{minipage}\\
\begin{minipage}{0.48\linewidth}\begin{center} (a) \end{center}\end{minipage}
\begin{minipage}{0.48\linewidth}\begin{center} (b) \end{center}\end{minipage}
\caption{ (a) The velocity profiles  at $t=0$ and
$t=1$ of  the one-dimensional Burgers equation with constant
diffusion coefficient $\nu=10^{-3}$. $n=2000$ grid points  are used
to obtain the full model for the fixed space domain $[0, 1]$,
while $k^0=100$ grids are used to obtain a coarse model. The first
10 Fourier modes are extracted to construct the trial solution.
When $\eta=10^{-10}$, it takes three iterations for SIRM to obtain an
accurate solution. (b) Convergence of SIRM for different $\eta$
and $k^0$ values. Plot of the maximal $L_2$ error,
$\|e^j\|_\infty=\sup\{\|\hat u^j(t)-u(t)\|:t\in[0,1]\}$,  between
the benchmark solution $u(t)$ and the iterative solution $\hat
 u^j(t)$. }
\label{fig:1dburgers}\vspace{-3mm}
\end{center}
\end{figure}

In the numerical simulation, the diffusion coefficient is given by $\nu=10^{-3}$. The full model is
obtained using  $n=2000$ grid points, while the trial solution is
obtained by extracting the first $10$ Fourier modes from a coarse
simulation with $k^0=100$ grid points. Because the one-dimensional
Burgers equation has a positive velocity, a wave will propagate to
the right with the higher velocities overcoming the lower
velocities and creating steep gradients. This steepening continues
until a balance with the dissipation is reached, as shown by the
velocity profile at $t=1$ in Figure \ref{fig:1dburgers}(a).
Because states of the Burgers equation have  high variability with
time evolution,
 more modes are necessary in order to present the whole solution trajectory
 with high accuracy. Meanwhile, the SIRM method requires more
iterations to obtain convergence.

The convergence plot for SIRM is shown in Figure
\ref{fig:1dburgers}(b).  Equation (\ref{num}) gives an adaptive
dimension, $k$, in each iteration: their values are 21, 38, and 60
for the first three iterations when $\eta=10^{-6}$;
  are 26, 49, and 85 when $\eta=10^{-8}$;  are  30,
  62, and 105 when $\eta=10^{-10}$; and  are 34, 76, and
129 when $\eta=10^{-12}$. When $\eta\le 10^{-8}$, the error of the
approximate solution decreases in the first few iterations and
then converges to a fixed value,  which is mainly determined by
the truncation error  produced by SVD. In order to achieve higher
resolution, for each iteration, more snapshots are needed for each
iteration to construct the information matrix and include more
modes in the associated reduced model.

What is more, the SIRM method can be used to estimate errors of
 other approximate models as well. The Euclidean distance between the actual
solution $u(t)$ and the approximate solution $\hat u^0(t)$ as a
function of $t$  can indicate the accuracy of a coarse model (or a
reduced model),
\begin{equation}
\|e^0(t)\|=\| u(t)-\hat u^0(t)\|.
\end{equation}
However, in many applications, the actual solution $u(t)$ is
unknown or very expensive to obtain. In this case, the SIRM method
can be used to obtain a more precise solution  $\hat u^1(t)$, and
the Euclidean distance between $\hat u^1(t)$ and $\hat u^0(t)$ can
be used as an error estimator,
\begin{equation}\label{deltat}
\|\Delta^0(t)\|=\|\hat u^1(t)-\hat u^0(t)\|.
\end{equation}
Although $\hat u^1(t)$ is only guaranteed to have higher accuracy
than $\hat u^0(t)$ locally, (\ref{deltat}) can be applied to
identify whether and when the trial solution has a significant
discrepancy from the actual solution. More generally, the error of
the iterative solution $\hat
 u^j(t)$   computed by SIRM,
\begin{equation}
\|e^j(t)\|=\|u(t)-\hat u^j(t)\|,
\end{equation}
 can (at least locally) be approximated by the difference
between $\hat u^{j+1}(t)$ and $\hat u^j(t)$, as follows:
\begin{equation}
\|\Delta ^j(t)\|=\|\hat u^{j+1}(t)-\hat u^j(t)\|.
\end{equation}
For this reason,  the criterion $\left \| \hat x^{j}- \hat x^{j-1}
\right \|_{\infty}<
    \epsilon$ is used in Algorithm \ref{alg:SIRM} to indicate convergence of SIRM.

Revisiting the  one-dimensional Burgers equation,  Figure
\ref{fig:burgerserr} shows that $\|\Delta ^j(t)\|$ is a good
approximation for  the actual error $\|e^j(t)\|$    for $t\in[0,1]$.

\begin{figure}
\begin{center}
\begin{minipage}{0.7\linewidth} \begin{center}
\includegraphics[width=1\linewidth]{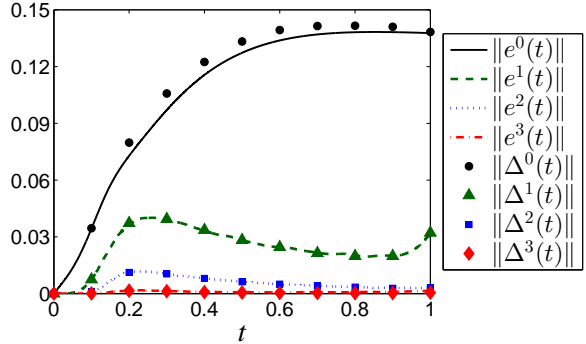}
\end{center} \end{minipage}\\
\caption{   Comparison  of the actual error
$\|e^j(t)\|=\|u(t)-\hat u^j(t)\|$ with the estimated error
$\|\Delta ^j(t)\|=\|\hat u^{j+1}(t)-\hat u^j(t)\|$  for
$t\in[0,1]$, where $u(t)$ is the actual solution of the
one-dimensional Burgers equation computed by $2000$ grid points,
$\hat u^0(t)$ is the trial solution obtained by extracting the
first 10 Fourier modes from a coarse simulation based on 100
grid points, and $\hat u^j(t)$ ($j\ne0$) are iterative solutions
computed by the SIRM method.} \label{fig:burgerserr}\vspace{-3mm}
\end{center}
\end{figure}

\section{Local SIRM}\label{sec:llp}
In the previous section, the presented analysis and simulations
illustrate that under certain conditions the SIRM method is able
to obtain a convergent solution in the global time domain.
However, the
 SIRM method still has existing redundancy with respect to both dimensionality and
computation, as described in the following, that could be
improved.

First, the reduced subspace formed by POD in SIRM keeps some
redundant dimensions of the original system in each iteration. To
explain this, consider  a large-scale dynamical system whose
solution exhibits vastly different states as it evolves over a
large time horizon. In order to obtain a highly accurate
representation for the entire trajectory, we need a subspace with
relatively high dimensionality to form a surrogate model. However,
projection-based model reduction techniques usually generate small
but full matrices from large (but sparse) matrices. Thus, unless
the reduced model uses significantly fewer  modes,  computing the
reduced model could potentially be more expensive than computing the
original one. Notice that the orbit of a dynamical system
(\ref{general}) is a one-dimensional curve; thus, it is desired
that  a local section of curve be embedded into a linear subspace of much
lower dimensionality.

Second,  for each iteration, SIRM requires that the entire trajectory
be computed from the initial time  to the final
 time of interest, $T$, which causes computational redundancy. As
 a variant of the Picard iteration,  the rate of convergence of  SIRM
 could be very slow for a nonlinear system with a large time domain.
 Under certain conditions, SIRM has a locally linear
 convergence. As inequality (\ref{eqbound}) indicates, when $t\in J_0$ the rate
 of convergence of $\|e^j\|_\infty-\chi$ is given
 by $Ka\exp(Ka/2)/2$. However,  as (\ref{ejej1}) suggests, we
 cannot guarantee that  SIRM could obtain a better global
 solution in each iteration. Meanwhile, if we have already
 obtained convergence at $t=a$ for some $0<a\le T$,  it would be a waste of
 computation to return to $t = 0$ for  the next iteration.

 Thus, it is preferable  to partition the entire time
 domain into several smaller subintervals, obtain a convergent solution for one subinterval, and then  move forward to the next.
A simple concept of time-domain partition was already introduced
in~\cite{DihlmannM:11a} in the context of the standard
POD-Galerkin method.
 As opposed to time domain partition, space domain partition
\cite{FarhatC:12a} and parameter domain partition
\cite{PateraAT:10a} also have be devolved to construct local
reduced models using partial snapshots from a precomputed
database.
 In this section, we combine
the idea of time domain partition with SIRM and propose a local
SIRM algorithm for model reduction. For each subinterval, the
resulting method constructs a convergent sequence of approximating
trajectories solved in subspaces of significantly lower
dimensionality. Convergence analysis and the relation of the presented
method with some other numerical schemes are also discussed. Then,
we demonstrate its effectiveness in an example of the
Navier--Stokes simulation of a lid-driven cavity flow problem.


\subsection{Algorithm of Local SIRM}\label{sec:alllp}
Suppose the entire time domain $J:=[0, T]$ is partitioned into $M$
smaller subintervals $J_1, ... , J_M$ with $J_i:=[t_{i-1},t_{i}]$.
We slightly abuse the notation and  denote the subinterval index
by the subscript $i$. Let $t_0=0$ and $t_{M}=T$, such that $J =
\cup _{i = 1}^M{J_i}$. At subinterval $J_i$, the local solution
trajectory approximately resides in a linear subspace $S_i$
 spanned by column vectors in $\Phi_i$. Let $\Phi_i$ be orthonormal;
 then  the reduced
equation formed by the Galerkin projection is given by
\begin{equation}\label{localgalerkin}
\dot z = \Phi_i^T f(t,\Phi_i z)
\end{equation}
for $t\in J_i$.  The SIRM method can be applied to approach a locally
invariant subspace and obtain a convergent solution $x(t)$ for
this subinterval. Specifically, we initially set a trial solution,
$\hat x^0(t)$ for $t\in J_i$. During iteration $j$, an extended
data ensemble, $\hat Y^{j-1}_i$, which contains a small number of
snapshots within the subinterval, is constructed and then served
to generate the empirical eigenfunctions of the
 subspace to be used in the next iteration
cycle.  After locally projecting the full model onto this subspace
and constructing a reduced model through (\ref{localgalerkin}),
the time integration is carried out in a low-dimensional space to
obtain
  an updated approximate solution. Once sufficient accuracy is achieved, one can move forward to the next subinterval.

 Suppose a convergent solution ${{{\hat x}}(t)}$  for   subinterval $J_{i-1}$ is
 obtained by the SIRM method. Then,  the ending state of $J_{i-1}$, $\hat x(t_{i-1})$, is the
 starting state of the next subinterval $J_{i}$.
 There are several options to estimate the trial solution
 $\hat x^0(t)$ for $t\in J_i$, and we just list a few here.
 One can simply set the trial solution  as a constant, which
 means $\hat x^0(t) = \hat x(t_{i-1})$ for $t \in J_{i}$ (although this is  inaccurate).
Alternatively, a coarse model can be used to obtain a rough
estimation of ${{\hat x}^0}(t)$. These two methods can also be
used for SIRM, as discussed in the previous section. The third
option is to use the time history of the solution trajectory to
obtain an initial estimation of the invariant subspace. Similar to
\cite{MarkovinovicR:06a,RapunM:10a}, one can assume that  the
solution for  subinterval $J_i$ approximately resides in the
invariant subspace of the previous subinterval. Thus, a set of
empirical eigenfunctions can be generated by SVD of the state
matrix or the information matrix formed by  snapshots in
$J_{i-1}$. Especially, if only the starting snapshot and the
ending snapshot are used to construct the initial information
matrix, we have
\begin{equation}\label{ab2}
\hat Y_i^0 = [\hat x (t _{i - 2}), \hat x (t _{i - 1}),\gamma f({t_{i -
2}},\hat x (t _{i - 2})),\gamma f({t_{i - 1}},\hat x (t _{i - 1}))].
\end{equation}
After projecting the full model onto this subspace, we can
calculate the trial solution for $t\in J_{i}$. Since we do not
have snapshots for $t<0$, the time-history-based initialization
cannot be used for the first subinterval.

 After obtaining a trial solution for a subinterval, SIRM is used
to obtain a better approximation of the actual solution. When the width of a
subinterval is small enough, the reduced equation has a
significantly lower dimension. Let
 $m$  denote the  number of sampling snapshots  in the whole
 trajectory and  $m'$ denote the number of sampling snapshots  within
 one time interval. For each $i$, both $\hat x(t_{i-1})$ and $\hat x(t_{i})$ are
 sampled for the extended data ensemble. Thus, $m=(m'-1)\times M+1$.
  If $m'=2$,  the information matrix
 $\hat Y_{i}^j$
can be constructed from snapshots at $t_{i-1}$ and $t_{i}$,
\begin{equation}\label{am2}
\hat Y_i^j = [\hat x (t _{i - 1}),\hat x^j({t_i}), \gamma f({t_{i -
1}},\hat x (t _{i - 1})),\gamma f({t_i},\hat x^j({t_i}))].
 \end{equation}
Then,   $\Phi^j_i$ can be constructed by the SVD.   When $m'$ is
small enough, say, $m'\le 5$, there is no need to further reduce
dimensions from $\hat Y_i^j$. Instead, $\Phi^j_i$ can be computed more efficiently
by the Gram--Schmidt process.  Algorithm \ref{llpalgorithm}
represents the complete process of the local SIRM method.

\medskip
\begin{algorithm}
\caption{Local SIRM} \label{llpalgorithm}
\begin{algorithmic}
\REQUIRE The initial value problem (\ref{general}).  \ENSURE An
approximate solution $\hat x(t)$.
 \STATE Divide the whole time
domain into smaller subintervals $J_1, ... , J_M$.

\FOR {subinterval $i$}
 \STATE Set  a test function $\hat x^0(t)$ as the
trial solution. \STATE Obtain a local solution by SIRM.
     \ENDFOR
\end{algorithmic}
\end{algorithm}

Using the formula (\ref{costrom}), we can obtain the computational
complexity for Algorithm \ref{llpalgorithm}. Table
\ref{tab:compare} illustrates the complexity of the full model,
the (global) SIRM method, and the local SIRM method.
Compared with the full model, the SIRM and local SIRM methods are more efficient only when the following conditions are satisfied: (1) The
standard POD-Galerkin approach is significantly faster than the
original model, and (2) the number of sampling
points $m$ is much smaller than the total number of time steps
$N_T$.

Next, we compare the computational complexity of SIRM and its variant, local SIRM.
 In order to achieve the same level of accuracy for the same problem,
 the number of iterations needed, $N_I$, for Algorithm \ref{alg:SIRM} is
 usually much greater than the average number of iterations needed, $N_I'$,
 for one subinterval of Algorithm \ref{llpalgorithm}. In addition, since $m \simeq Mm'$, we can assume $k \simeq
Mk'$. Although there is no general formula for $\beta(k,n)$, we
may expect that it is at least a linear function of $k$, and
therefore $\beta(k,n)\ge M\beta(k',n)$ holds. In fact, SIRM can be
considered   a special case of the local SIRM method where
$M=1$, and the local reduced model offers more flexibility to
choose a subspace dimension. Furthermore, the unit step of
(\ref{localgalerkin}) could be the same as the unit step of the full
model $\delta t$ when computing the time integration. Suppose  SVD or the time integration
plays a dominant role in determining the complexity of Algorithm
\ref{alg:SIRM};   a local reduced model can obtain at least $M$
times speedups.

\begin{table}[htbp]
\begin{center}
\caption{ Complexity of the full model,  SIRM, and local SIRM using
 implicit schemes for time integration}
  \label{tab:compare}
\begin{tabular}{|l|l|}
\hline
 Full model   & $ N_T\cdot\left(b\gamma(n)/5+b^2n/20\right)$\\
  \hline
  SIRM  & $N_I\cdot\left(4m^2n+\beta (k,n)+N_T\cdot\left(k \hat \gamma
(k,n)/5+k^3/15\right)+m\gamma(n)+mnk\right)$\\
  \hline
  Local SIRM  &  $N_I'\cdot\left(4mm'n+M\beta (k',n)+N_T\cdot\left(k' \hat \gamma
(k',n)/5+k'^3/15\right)+m\gamma(n)+mnk'\right)$\\
 \hline
\end{tabular}
\end{center}
\end{table}

Based on the aforementioned complexity analysis, we discuss some heuristics for  some parameter choice strategies of the local SIRM method. Although the selection of  $m$ has some flexibility, a good choice of $m$ should balance  accuracy and computational speed. Once $m$ is determined, in order to generate
maximal speedups for one iteration, each subinterval contains a small number of sampling snapshots, say, $m'=2$ or $m'=3$. Numerical study in section 4.3 indicates that if $m$ remains constant, a large $m'$ value cannot significantly increase the accuracy for the lid-driven cavity flow problem.
If the Gram--Schmidt process is used to form a set
of orthonormal  eigenvectors, the dimension of each local subspace, $k'$, can be directly determined
by  $m'$.  Usually, $k'=2m'$ if the solution trajectory is represented by one curve.
 The multiplier 2 stems  from the fact that  the information matrix contains both state vectors and  their
corresponding tangent vectors. For the lid-driven cavity flow problem, since the solution  involves both $\psi$ and $\omega$, we have $k'=4m'$.

\subsection{Convergence Analysis}\label{sec:anllp}
We have already shown the capability of SIRM to effectively
approach a globally invariant subspace for a dynamical system, and
thus generate a sequence of functions that converges to the actual
solution. As an extension of SIRM, local SIRM generates a set of
local invariant subspaces and obtains corresponding local
solutions. The union of all these local solutions forms a full
trajectory for the original system.

We begin here with the first subinterval. In an ideal situation,
$x_0 \in S_1^j$, while the state and the vector field satisfy
$\hat x^{j-1}(t) \subset S_1^j$ and $f(t,\hat x^{j-1}(t)) \subset
S_1^j$ for all $t\in J_1$, respectively. As Theorem
\ref{globalconvergence} indicates, the sequence $\{\hat x^j(t)\}$ generated by the local SIRM method approaches $x(t)$
for $t\in J_1$. If the vector field is
Lipschitz, then the local solution to (\ref{general}) on subinterval
  $J_2$ continually depends on the initial condition $x(t_1)$.
For this reason, starting from $\hat x (t_{1})$, we can obtain a sequence
of functions that converges to $x(t)$ for $t\in J_2$. We can then
move forward to the rest of the subintervals and achieve the
following theorem.
\medskip

\begin{theorem}[{\rm  convergence of local SIRM}]\label{llpconvergence}
 Consider solving the initial value problem
\emph{(\ref{general})} by local SIRM for the time domain $J=[0,
T]$, which is partitioned into $M$ smaller subintervals $J_1, ...
, J_M$ with $J_i:=[t_{i-1},t_{i}]$. Suppose $f(t,x)$ is a locally
Lipschitz function of $x$ and a continuous function of $t$  for
all $(t, x) \in J\times \mathcal{D'}$, where $\mathcal{D'}$ is an
open set that contains $x(t)$ for all $t\in J$. For subinterval
$J_i$, the SIRM method is applied to obtain an approximation for the
local solution. Let $x(t)$ be the local solution of the full
model,  and let $\hat x^j(t)$ be the solution of the reduced model at
iteration $j$. For each iteration, the reduced subspace $S^j_i$
contains $\hat x (t_{i-1})$.
Furthermore, the vector field satisfies $f(t,\hat x^{j-1}) \subset
S^j_i$ for all $t\in J_{i}$. Then, for all $i \in \{1,...,M\}$ and
$t \in J_i$, the sequence of functions $\{\hat x^j(t)\}$ uniformly
converges to $x(t)$.
\end{theorem}
\medskip

Finally, it is interesting to consider the local SIRM method as a
generalization of many current time integration schemes. We can
again consider $m'=2$ as an example. The
time-history initialization provides a linear subspace spanned by
$\hat Y_i^0$ to estimate the trial solution for $t\in J_i$.
Especially,  when $t=t_i$, the initial estimation of the state
vector is given by
\begin{equation} \hat
x^0({t_i}) = \hat Y_i^0 \cdot \varsigma _i^0,
\end{equation}
where $\hat Y_i^0$ is given by (\ref{ab2}) and $\varsigma _i^0$ is
a  vector that contains four elements. Suppose the  width of each
subinterval equals $\delta T$ and the width of one time step of
integration equals $\delta t$. As  $\delta T \to \delta t$, local
SIRM degenerates to the two-step Adams--Bashforth scheme if
\begin{equation*}
\varsigma _i^0=\left [0, 1, -\frac{\delta t}{2\gamma },
\frac{3\delta t}{2 \gamma} \right]^T.
\end{equation*}
 On the other hand, if one uses the SIRM method to obtain a better estimation
 at $t=t_i$, the approximate solution is given by
\begin{equation}
\hat x^1({t_i}) = \hat Y_i^1 \cdot \varsigma _i^1,
\end{equation}
where it is assumed that only two snapshots are used to construct
the information matrix $\hat Y_i^1$, as expressed by (\ref{am2}).
As $\delta T \to \delta t$, local SIRM degenerates to the
Crank--Nicolson scheme if
\begin{equation*}
\varsigma _i^1=\left [1, 0, \frac{\delta t}{2\gamma}, \frac{\delta t}{2
\gamma } \right]^T.
\end{equation*}
More generally, suppose $m'$ snapshots are sampled from the
previous subinterval. Then, as  $\delta T \to \delta t$, the
time-history initialization can degenerate to the $m'$-step
Adams--Bashforth method if proper coefficients are set for
$\varsigma _i^0$. In addition, if the $Y_i^j$ has $m'+1$ snapshots
from $J_i$, and the first $m'$ snapshots are overlapping with
$J_{i-1}$, then each iteration defined by SIRM can degenerate to
the $m'$-step Adams-Moulton  method. Furthermore, if $\delta T
=m'\delta t$, then each iteration defined by SIRM is a generalized
form of the $m'$-order Runge--Kutta  method with variable
coefficients.

 However, as a manifold learning approach, local SIRM  applies reduced models to determine the
coefficient values for each subinterval. This is more flexible
than a common scheme for   time integration because the latter
uses predesigned coefficients for each column  of the information
matrix. Therefore, the local SIRM method has the ability to
provide more stable results for a fixed time interval. Even if
$\delta T \gg \delta t$, local SIRM can still generate stable
results with high accuracy. In the next subsection, the local SIRM
approach is applied to a lid-driven cavity flow problem.

\subsection{Cavity Flow Problem}\label{sec:CFllp}
Consider a lid-driven cavity flow problem in a rectangular domain
$\Omega=[0,1] \times [0,1] $. The space domain is fixed in time.
Mathematically, the problem can be represented in terms of the
stream function $\psi$ and vorticity $\omega$ formulation of the
 incompressible Navier--Stokes equation.  In nondimensional form,  the
governing equations are given as
\begin{equation} \label{sf}
\psi_{xx} + \psi_{yy} =  - \omega ,
\end{equation}\vspace{-5mm}
\begin{equation} \label{vt}
\omega_t =  - \psi_y  \omega_x + \psi_x\omega_y +
\frac{1}{\rm{Re}}\left( \omega_{xx} + \omega_{yy} \right),
\end{equation}
where Re is the Reynolds number and $x$ and $y$ are the
Cartesian coordinates. The velocity field is given by $u=\partial
\psi /\partial y$, $v=-\partial \psi /\partial x$.
No-slip boundary conditions are applied on all nonporous walls including the top wall moving at speed $U=1$. Using Thom's formula~\cite{ThomA:33a}, these conditions are, then, written in terms of stream function and vorticity. For example on the top wall one might have
\begin{equation}
\psi_B=0,
\end{equation}
\begin{equation}
\omega_B=\frac{-2\psi_{B-1}}{h^2}-\frac{U}{h},
\end{equation}
where  subscript $B$ denotes points on the moving wall, subscript $B-1$ denotes points adjacent to the moving wall, and $h$ denotes grid spacing. Expressions for $\psi$ and $\omega$ at remaining walls with $U=0$ can be obtained in an analogous manner.
 The initial condition is set
as $u(x,y)=v(x,y)=0$. The discretization is performed on a uniform
mesh with   finite difference approximations. For the time integration of
(\ref{vt}), the implicit Crank--Nicolson scheme is applied
 for the diffusion term, and the explicit two-step
Adams--Bashforth method is employed for the advection term.

In the numerical simulation, the Reynolds number is given by
Re$=1000$. The full model uses $129 \times 129$ grid points and
$\delta t= 5\times10^{-3}$ as a unit time step.   The whole time
domain, $[0, 50]$, is divided into 250 subintervals. For each
subinterval, the trial solution is obtained through a simulation
based on $33 \times 33$ coarse grid points with a unit time step
of $4\delta t$. The same discretization scheme  is applied for the
coarse model. Thus,  the coarse model can cost less than  $1/64$ of the required operations in the full model.
 A sequence of functions defined by local SIRM is
used to approach the local solution.

\begin{figure}
\begin{center}
\begin{minipage}{0.48\linewidth} \begin{center}
\includegraphics[width=1\linewidth]{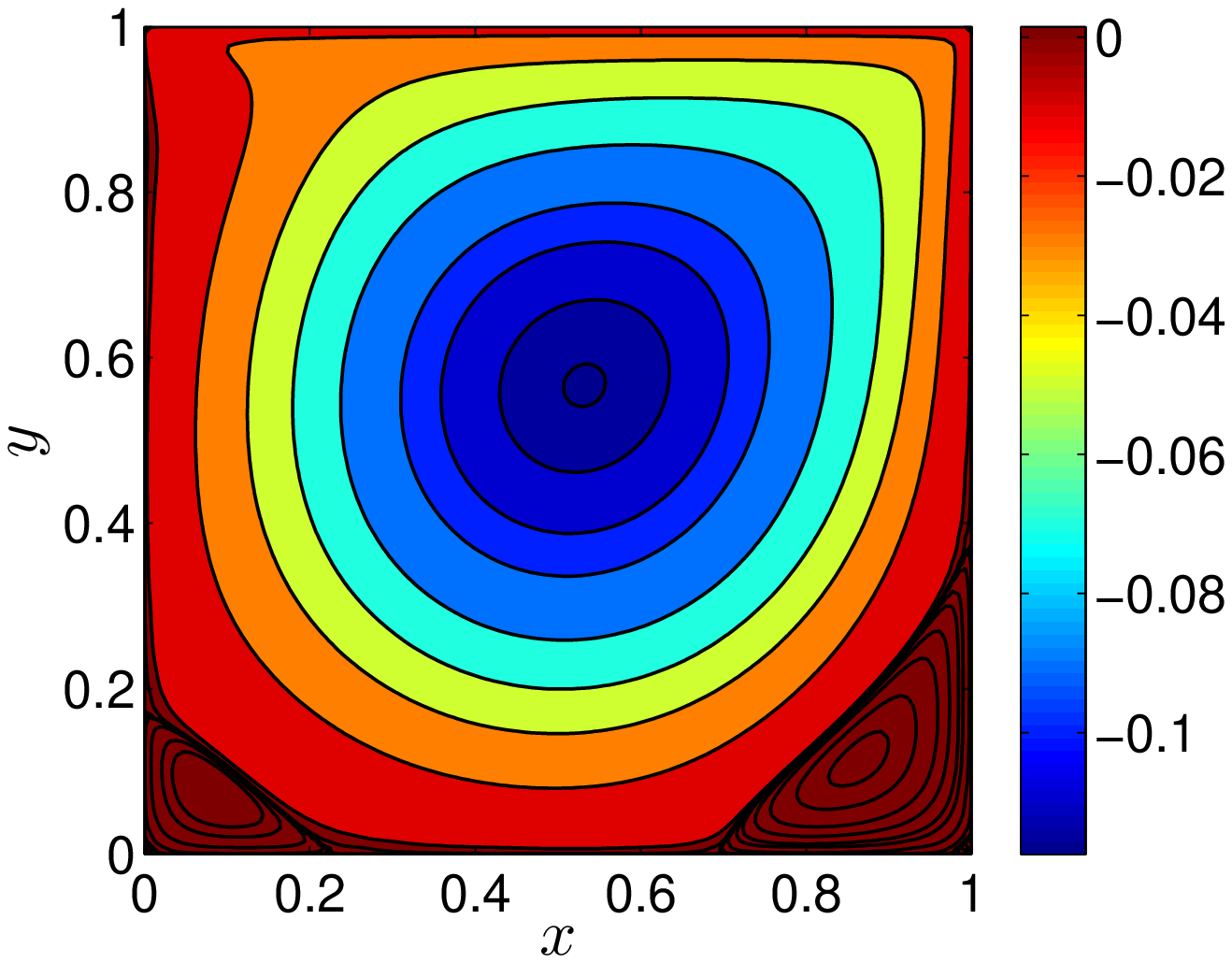}
\end{center} \end{minipage}
\begin{minipage}{0.48\linewidth} \begin{center}
\includegraphics[width=1\linewidth]{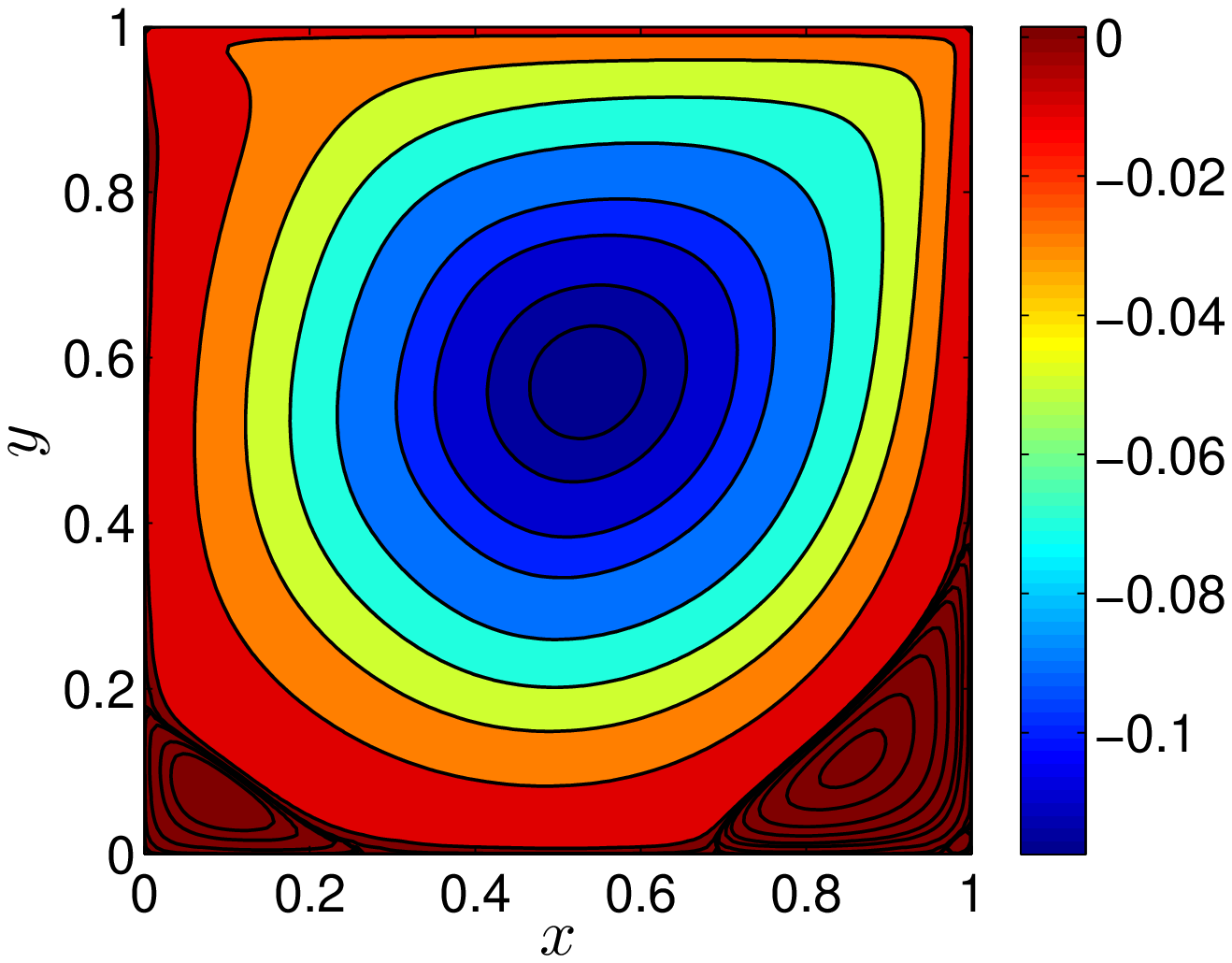}
\end{center} \end{minipage}
\begin{minipage}{0.48\linewidth}\begin{center} (a) \end{center}\end{minipage}
\begin{minipage}{0.48\linewidth}\begin{center} (b) \end{center}\end{minipage}
\caption{ Streamline pattern for driven cavity problem with
Re=$1000$.  (a) The full model uses  $129 \times 129$ grid points.
(b) The approximating result obtained through local SIRM. The
whole time domain, $[0, 50]$, is partitioned into 250
subintervals. For each subinterval, a trial solution is calculated
from $33 \times 33$ grid points. An average of five iterations are
used to achieve a better approximation. We plot the contours of
$\psi$ whose values are $-1\times 10^{-10}$, $-1\times 10^{-7}$,
$-1\times 10^{-5}$, $-1\times 10^{-4}$, $-0.01$, $-0.03$, $-0.05$,
$-0.07$, $-0.09$, $-0.1$, $-0.11$, $-0.115$, $-0.1175$, $1\times
10^{-8}$, $1\times 10^{-7}$, $1\times 10^{-6}$, $1\times 10^{-5}$,
$5\times 10^{-5}$, $1\times 10^{-4}$, $2.5\times 10^{-4}$,
$1\times 10^{-3}$, $1.3\times 10^{-3}$, and $3\times 10^{-3}$.}
\label{fig:stream}
\end{center}
\end{figure}

The streamline contours for the lid-driven cavity flow are shown
in Figure \ref{fig:stream}. In \ref{fig:stream}(a), the full model
matches well with the numerical results from~\cite{Ghia:82a}, and
the values of $\psi$ of the contours are the same as
shown in Table III of~\cite{Ghia:82a}. Local SIRM provides an
approximate solution. The main error occurs around the vortex
center, where the contour of $\psi=-0.1175$ is missing in
\ref{fig:stream}(b).
\begin{figure}
\begin{center}
\begin{minipage}{0.45\linewidth} \begin{center}
\includegraphics[width=1\linewidth]{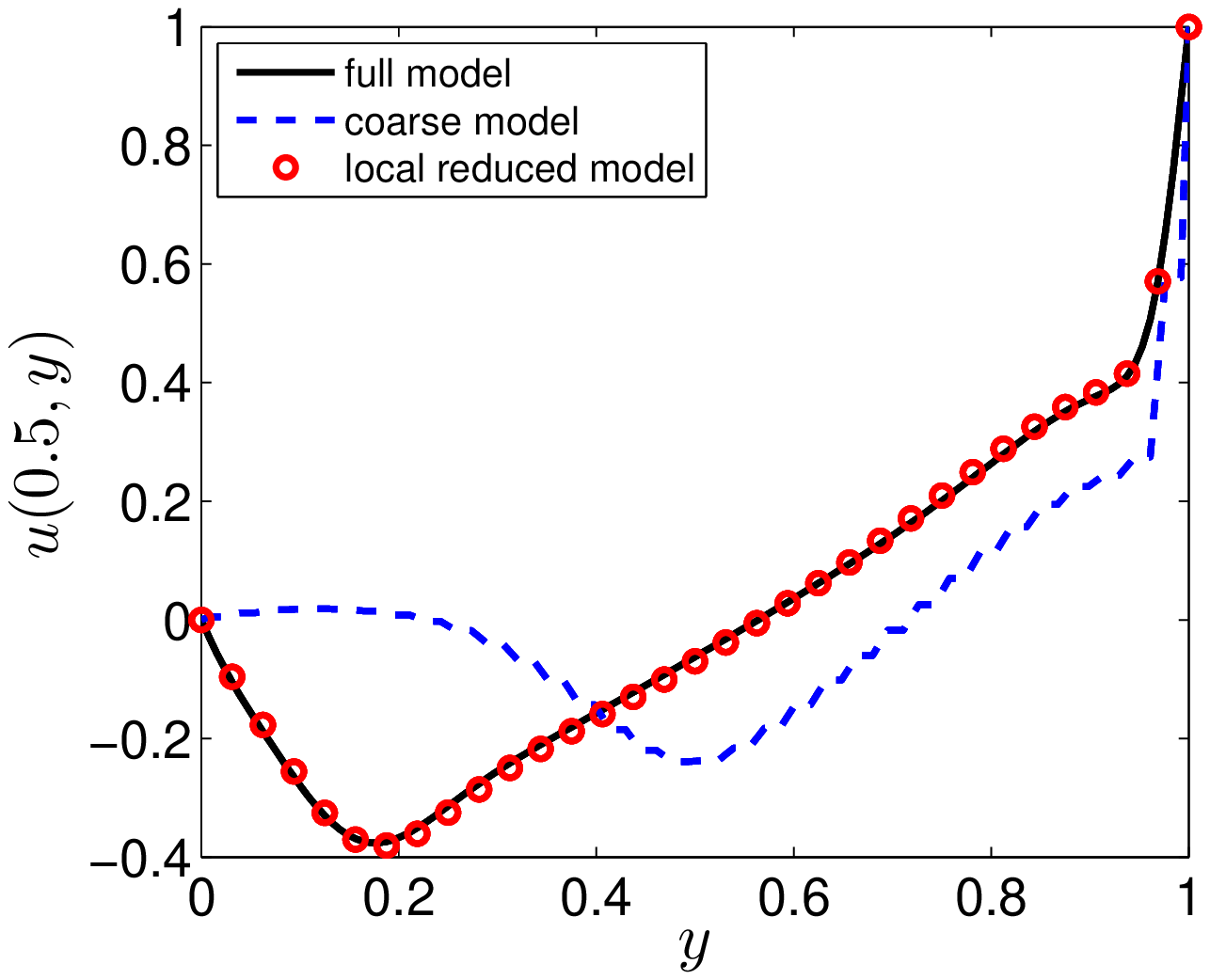}
\end{center} \end{minipage}
\begin{minipage}{0.45\linewidth} \begin{center}
\includegraphics[width=1\linewidth]{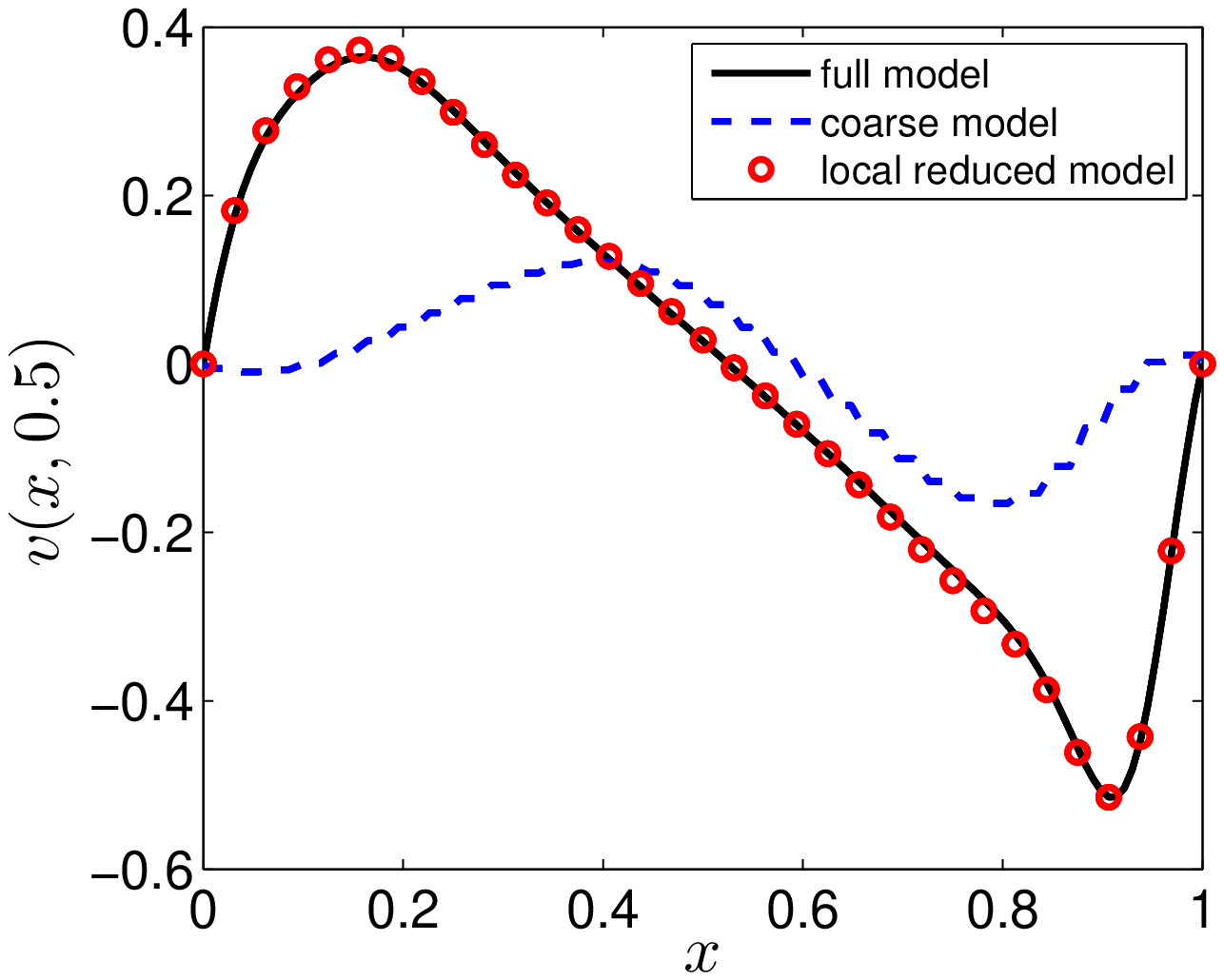}
\end{center} \end{minipage}
\begin{minipage}{0.48\linewidth}\begin{center} (a) \end{center}\end{minipage}
\begin{minipage}{0.48\linewidth}\begin{center} (b) \end{center}\end{minipage}
\caption{ (a) Comparison of the velocity component $u(x=0.5, y)$
along the $y$-direction passing the geometric center between the
full model, the coarse model, and the local SIRM method at $t=50$.
(b) Comparison of the velocity component $v(x, y=0.5)$ along the
$x$-direction passing geometric center between the full model, the coarse model,
and the local SIRM method at $t=50$. } \label{fig:velocity}
\end{center}
\end{figure}

\begin{figure}
\begin{center}
\begin{minipage}{0.48\linewidth} \begin{center}
\includegraphics[width=1\linewidth]{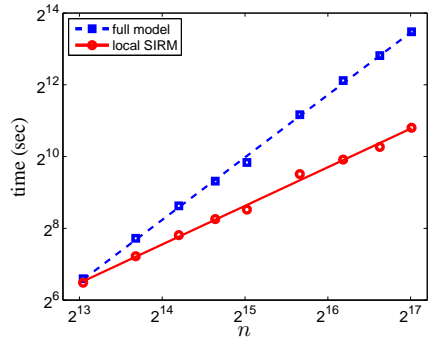}
\end{center} \end{minipage}\\
\caption{   Comparison  of the computational time
of the full model and the local SIRM method for the lid-driven
cavity flow problem. As the resolution increases from $65\times
65$ to $257\times 257$, the dimension of the full system, $n$,
increases from $2\times 65^2$ to $2\times 257^2$. Using a log-log plot,
 the asymptotic complexity can be determined by the linear regression coefficient.}
\label{fig:time}\vspace{-3mm}
\end{center}
\end{figure}


Figure \ref{fig:velocity} shows the velocity profiles for $u$
along the vertical line and $v$ along the horizontal line passing
through the geometric center of the cavity.  The coarse model
provides a trial solution, which significantly deviates from the
actual one. Then,  local SIRM  is used to obtain much more
accurate results. For each iteration, three snapshots and their
corresponding tangent vectors are used to form the information
matrix.  Instead of  POD, the Gram--Schmidt process is
applied here to form a set of orthonormal empirical eigenfunctions.
Since the local data ensemble contains both $\psi$ and $\omega$ as well as the associated tangent vectors, the
subspace dimension  is $12$.
An average of five
iterations are carried out to obtain a local convergent solution
for each subinterval.

  Since an explicit scheme is used
for the advection term, the CFL condition, $u \delta t/ \delta x+v \delta t/ \delta y \le
1$  is a necessary condition for stability. Therefore, if the
number of grid points increases from $65\times 65$ to $257 \times
257$, the unit time step decreases from $10^{-2}$ to $2.5\times
10^{-3}$ accordingly. Accounting for this, the asymptotic computational
complexity of the full model for the entire time domain is no less than
$\mathcal{O}(n^{1.5})$.   The above analysis only focuses on the
 advection term. Since the diffusion term uses an implicit
 scheme, there is no extra limit to the unit time step for the stability requirement. However,  a large  $n$ will lead to a
 slower convergence for many iterative methods, such as the successive over-relaxation method or the conjugate gradient method. Thus,
  $\mathcal{O}(n^{1.5})$ provides only a low bound
 estimation for the full model.

Since the Navier--Stokes
equation  contains only linear and quadratic terms, the complexity
of the reduced model constructed by the Galerkin projection for
one-step integration does not explicitly depend on $n$.
Moreover, the computational complexity of all the other terms  of local SIRM  in Table \ref{tab:compare}
 depends at most linearly on $n$.  Thus, we may roughly estimate that the overall complexity of
 local SIRM  is $\mathcal{O}(n)$. Figure \ref{fig:time} compares the running time of the full model
and the running time of  SIRM for different resolutions in
 (\ref{sf}) and (\ref{vt}). Except $n$ and $\delta t$, all the  parameters
remain the same. The linear regression indicates that the
asymptotic complexity of the full model is $\mathcal{O}(n^{1.74})$, and
the asymptotic complexity of the reduced model is
$\mathcal{O}(n^{1.07})$ using the same scheme.


\begin{table}[htbp]
\begin{center}
\caption{The maximal $L_2$ error between the benchmark solution
and  approximate solutions solved by the local SIRM method for
different $m$ and $m'$ values.}
 \label{tab:mm}
\begin{tabular}{|l|c|c|c|c|c|}
\hline
  & $m'=2$ &  $m'=3$ & $m'=5$ & $m'=6$ & $m'=10$\\
  \hline
  $m=500$ & 0.5689 & 0.5499 & 0.5407& 0.5411 &0.5403 \\
$m=250(m'-1)$ & 1.5029 & 0.5499 & 0.2335 & 0.1930 & 0.1175 \\
 \hline
\end{tabular}
\end{center}
\end{table}

Finally, Table \ref{tab:mm} shows the  maximal $L_2$ error for the
local SIRM method using different $m$ and $m'$ values. If each
local reduced equation is solved in a larger subinterval with more modes while the total number
of sampling snapshots remains the same, there is no significant
improvement in accuracy. On the other hand, if the length of each subinterval remains the same but we sample
more snapshots, a more accurate solution
can be achieved. Thus, a good $m$ value should balance  accuracy and  cost of the reduced model, while a small $m'$ is desired for  the lid-driven cavity flow problem.

\section{ Conclusion}\label{sec:conclusion}
In this article,  a new online manifold learning framework,
subspace iteration using reduced models (SIRM), was proposed for
the reduced-order modeling of large-scale nonlinear problems where
both the data sets and the dynamics are systematically reduced.
This framework does not require prior simulations or experiments
to obtain  state vectors.  During each iteration cycle, an
approximate solution is calculated in a low-dimensional subspace,
  providing many snapshots to construct an information matrix.
The POD (SVD) method could be applied to generate a set of
empirical eigenfunctions that span a new subspace. In an ideal
case, a sequence of functions defined by SIRM uniformly converges
to the actual solution  of the original problem. This article also
discussed the truncation error produced by  SIRM  and
provided an error bound. The capability of SIRM to solve a high-dimensional system with high accuracy was demonstrated in several
linear and nonlinear equations. Moreover, SIRM could also be used
as a posterior error estimator for  other coarse or reduced
models.

In addition, the local SIRM method was developed as an extension that can reduce
the cost of  SIRM. The SIRM method is used to obtain  a
better approximate solution for each subinterval of a partitioned
time domain. Because each subinterval has  less state variation,
the associated reduced model could be small enough. The
numerical results of the nonlinear Navier--Stokes equation through
a cavity flow problem implied that the local SIRM method could
obtain significant speedups for a large-scale problem while maintaining good accuracy.

There are some interesting open questions to study in the future.
For example, since the choice of the extended data ensemble is not
unique, there might be other methods that can be used  to form an
information matrix that results in a more efficiently reduced
model. It should be noted that the POD-Galerkin approach is not
the only technique that can be used to extract the dominant modes
from an information matrix
 and to construct a reduced model. How to combine SIRM with other model reduction techniques that exhibit higher efficiency remains a topic for future research.

\bibliographystyle{siam}
\bibliography{RefA3}

\end{document}